\documentclass[11pt,letterpaper,twoside,reqno,nosumlimits]{amsart}

\usepackage{fixltx2e} %

\usepackage[usenames,dvipsnames]{xcolor}
\usepackage{fancyhdr}
\usepackage{amsmath,amsfonts,amsbsy,amsgen,amscd,mathrsfs,amssymb,amsthm}
\usepackage{subcaption}
\usepackage{url}

\usepackage{mathtools}

\mathtoolsset{showonlyrefs}

\usepackage[font=small,margin=0.25in,labelfont={sc},labelsep={colon}]{caption}

\usepackage{tikz}
\usepackage{microtype}
\usepackage{enumitem}

\definecolor{dark-gray}{gray}{0.3}
\definecolor{dkgray}{rgb}{.4,.4,.4}
\definecolor{dkblue}{rgb}{0,0,.5}
\definecolor{medblue}{rgb}{0,0,.75}
\definecolor{rust}{rgb}{0.5,0.1,0.1}

\usepackage[colorlinks=true]{hyperref}

\hypersetup{urlcolor=rust}
\hypersetup{citecolor=dkblue}
\hypersetup{linkcolor=dkblue}

\usepackage{setspace}

\usepackage{graphicx}
\usepackage{booktabs,longtable,tabu} %
\usepackage{multirow} %

\usepackage{float}

\usepackage{fourier}
\usepackage{bm}

\graphicspath{{art/}}

\newtheorem{theorem}{Theorem}[section]

\newtheorem{proposition}[theorem]{Proposition}

\theoremstyle{definition}

\newtheorem{remark}[theorem]{Remark}

\newtheorem{mywarning}[theorem]{Warning}

\numberwithin{equation}{section} 

\floatstyle{plaintop}
\newfloat{recipe}{thp}{lor}
\floatname{recipe}{Recipe}

\usepackage[noend]{algpseudocode}
\usepackage{algorithm,algorithmicx}

\algrenewcommand\alglinenumber[1]{\sf\footnotesize\color{gray}{#1}}
\algrenewcommand\algorithmicrequire{\textbf{Input:}}
\algrenewcommand\algorithmicensure{\textbf{Output:}}

\providecommand{\mathbold}[1]{\bm{#1}}  %

\renewcommand{\phi}{\varphi}

\newcommand{\eps}{\varepsilon}

\newcommand{\half}{\tfrac{1}{2}}

\newcommand{\econst}{\mathrm{e}}

\newcommand{\Id}{\mathbf{I}}

\newcommand{\coll}[1]{\mathscr{#1}}

\providecommand{\mathbbm}{\mathbb} %

\newcommand{\R}{\mathbbm{R}}
\newcommand{\C}{\mathbbm{C}}
\newcommand{\N}{\mathbbm{N}}
\newcommand{\Z}{\mathbbm{Z}}

\newcommand{\abs}[1]{\left\vert {#1} \right\vert}

\newcommand{\diff}[1]{\mathrm{d}{#1}}
\newcommand{\idiff}[1]{\, \diff{#1}}

\newcommand{\Prob}[1]{\mathbbm{P}\left\{{#1}\right\}}

\newcommand{\Expect}{\operatorname{\mathbb{E}}}

\newcommand{\normal}{\textsc{normal}}

\newcommand{\vct}[1]{\mathbold{#1}}
\newcommand{\mtx}[1]{\mathbold{#1}}

\newcommand{\lspan}[1]{\operatorname{span}{#1}}

\newcommand{\range}{\operatorname{range}}

\newcommand{\srank}{\operatorname{srk}}
\newcommand{\rank}{\operatorname{rank}}

\newcommand{\diag}{\operatorname{diag}}

\newcommand{\norm}[1]{\left\Vert {#1} \right\Vert}

\allowdisplaybreaks

\title[Randomized Block Krylov Methods for Extreme Eigenvalues]{Randomized Block Krylov Methods \\
for Approximating Extreme Eigenvalues}

\author[J.~A.~Tropp]{Joel~A.~Tropp}

\date{30 September 2017.  Revised 26 February 2018 and 1 September 2020 and 1 October 2021.}


\begin{document}

\begin{abstract}
Randomized block Krylov subspace methods form a powerful class of algorithms for computing
the extreme eigenvalues of a symmetric matrix or the extreme singular values of a general matrix.
The purpose of this paper is to develop
new theoretical bounds on the performance of randomized block Krylov subspace methods for these problems.
For matrices with polynomial spectral decay, the randomized block Krylov method can
obtain an accurate spectral norm estimate using only a \emph{constant}
number of steps (that depends on the decay rate and the accuracy).  Furthermore, the analysis reveals
that the behavior of the algorithm depends in a delicate way on the block size.
Numerical evidence confirms these predictions.
\end{abstract}

\maketitle

\section{Motivation and Main Results}

Randomized block Krylov methods have emerged as a powerful tool for
spectral computation and
matrix approximation~\cite{RST09:Randomized-Algorithm,HMT11:Finding-Structure,MRT11:Randomized-Algorithm,HMST11:Algorithm-Principal, MM15:Randomized-Block,WZZ15:Improved-Analyses,DIKM18:Structural-Convergence,DI19:Low-Rank-Matrix,YGL18:Superlinear-Convergence,MT20:Randomized-Numerical}.
At present, our understanding of these methods is more rudimentary
than our understanding of simple Krylov subspace methods in either the deterministic
or random setting.  %
Our aim is to develop \textit{a priori} bounds that better explain the remarkable
performance of randomized block Krylov methods.

This paper focuses on the most basic questions. How well can we
estimate the maximum (or minimum) eigenvalue of a symmetric matrix 
using a randomized block Krylov method?
How well can we estimate the maximum (or minimum) singular value of a general matrix?
We present the first detailed analysis that addresses these questions.  The argument streamlines
and extends the influential work~\cite{KW92:Estimating-Largest} of Kuczi{\'n}ski \& Wo{\'z}niakowski
on the simple Krylov method with one randomized starting vector.  We discover that increasing
the block size offers some clear advantages over the simple method.

Randomized block Krylov methods have also been promoted as a tool
for low-rank matrix approximation.  Theoretical analysis of these
algorithms appears in several recent
papers~\cite{MM15:Randomized-Block,DIKM18:Structural-Convergence,DI19:Low-Rank-Matrix,YGL18:Superlinear-Convergence}.
In an unpublished companion paper~\cite{Tro18:Analysis-Randomized-TR},
we complement these works with a more detailed analysis that offers
several new insights.  The mathematical approach is different from
the case of extreme eigenvalues, and it involves additional ideas
from the analysis of the randomized SVD algorithm~\cite{HMT11:Finding-Structure}.
In collaboration with Robert J.~Webber, we are currently preparing
a review article~\cite{TW21:Randomized-Block} with a definitive treatment of randomized
block Krylov methods for low-rank matrix approximation.

\subsection{Block Krylov Methods for Computing the Maximum Eigenvalue}

Let us begin with a mathematical description of a block Krylov method for estimating
the maximum eigenvalue of a symmetric matrix.  See Section~\ref{sec:implementation}
for a brief discussion about implementations.

\subsubsection{Block Krylov Subspaces}

Suppose that we are given a symmetric matrix $\mtx{A} \in \R^{n \times n}$.
Choose a \emph{test matrix} $\mtx{B} \in \R^{n \times \ell}$, where the
number $\ell$ is called the \emph{block size}.  Select a parameter
$q \in \N$ that controls the \emph{depth} of the Krylov subspace.
In concept, the block Krylov method constructs the matrix
\begin{equation} \label{eqn:krylov-matrix}
\mtx{S}_q(\mtx{A}; \mtx{B}) := \begin{bmatrix}
	\mtx{B} & \mtx{AB} & \mtx{A}^2 \mtx{B} & \dots & \mtx{A}^q \mtx{B}
\end{bmatrix}
	\in \R^{n \times (q+1)\ell}.
\end{equation}
The range $K_q$ of the matrix $\mtx{S}_q$ is called a \emph{block Krylov subspace}:
\begin{equation} \label{eqn:krylov-subspace}
K_q(\mtx{A}; \mtx{B}) := \range(\mtx{S}_q) \subset \R^{n}.
\end{equation}
The block Krylov subspace admits an alternative representation in terms of polynomials:
\begin{equation} \label{eqn:krylov-poly}
K_q(\mtx{A}; \mtx{B}) = \operatorname{span}\big\{ \range\big( \phi(\mtx{A}) \mtx{B} \big) : \phi \in \mathcal{P}_q \big\},
\end{equation}
where $\mathcal{P}_q$ is the linear space of real polynomials with degree at most $q$.

\begin{mywarning}[Notation]
Be aware that our notation differs slightly from the traditional usage.
In $\mtx{S}_q(\mtx{A}; \mtx{B})$ and in $K_q(\mtx{A}; \mtx{B})$, the highest
matrix power is $\mtx{A}^q$, rather than the usual $\mtx{A}^{q - 1}$.
\end{mywarning}

\subsubsection{Invariance Properties of Krylov Subspaces}
\label{sec:krylov-invar}

Krylov subspaces have remarkable invariance properties that help explain their computational value.

\begin{itemize}
\item	The block Krylov subspace only depends on the range of the test matrix:
$$
K_q(\mtx{A}; \mtx{BT})
	= K_q(\mtx{A}; \mtx{B})
	\quad\text{for all nonsingular $\mtx{T} \in \R^{\ell \times \ell}$.}
$$

\item	The block Krylov subspace co-varies with the orientation of the matrices:
$$
K_q(\mtx{UAU}^*; \mtx{UB})
	= \mtx{U} K_q(\mtx{A}; \mtx{B}) 
	\quad\text{for all orthogonal $\mtx{U} \in \R^{n \times n}$.}
$$

\item	The block Krylov subspace is invariant under ``affine'' transformations of the input matrix:
$$
K_q(\alpha \mtx{A} + \beta \Id; \mtx{B})	
	= K_q(\mtx{A}; \mtx{B})
	\quad\text{for all $\alpha, \beta \in \R$.}
$$
\end{itemize}

\noindent
These facts follow directly from the definition~\eqref{eqn:krylov-matrix}--\eqref{eqn:krylov-subspace}
of the Krylov subspace and the representation~\eqref{eqn:krylov-poly} using polynomials.
For example, see~\cite[Sec.~12.2.2]{Par98:Symmetric-Eigenvalue} for the simple case $\ell = 1$.

\subsubsection{Computing Maximum Eigenvalues}

Block Krylov subspaces support a wide range of matrix computations.
The core idea is to compress the input matrix to the Krylov subspace
and to perform calculations on the (small) compressed matrix.  In other
words, Krylov methods belong to the class of Ritz--Galerkin methods;
see~\cite{Lan50:Iteration-Method,Pai71:Computation-Eigenvalues,Saa80:Rates-Convergence,Par98:Symmetric-Eigenvalue,Saa11:Numerical-Methods}.

In particular, we can obtain an estimate $\xi_{\max}(\mtx{A}; \mtx{B}; q)$
for the maximum eigenvalue $\lambda_{\max}(\mtx{A})$ of the input matrix
by maximizing the Rayleigh quotient of $\mtx{A}$ over the block Krylov subspace $K_q(\mtx{A}; \mtx{B})$:
\begin{equation} \label{eqn:eigenvalue-estimate}
\xi_{\max}(\mtx{A}; \mtx{B}; q)
	:= \max_{\vct{v} \in K_q(\mtx{A}; \mtx{B})} \frac{\vct{v}^* \mtx{A} \vct{v}}{\vct{v}^*\vct{v}}.
\end{equation}
The symbol ${}^*$ denotes the transpose of a matrix or vector,
and we instate the convention that $0/0 = 0$.  We may suppress the
dependence of $\xi_{\max}$ on $\mtx{A}$, $\mtx{B}$, or $q$ when
they are clear from context.

The Rayleigh--Ritz theorem~\cite[Cor.~III.2.1]{Bha97:Matrix-Analysis} implies that
the maximum eigenvalue estimate~\eqref{eqn:eigenvalue-estimate} always
satisfies
\begin{equation} \label{eqn:underestimate}
\lambda_{\min}(\mtx{A}) \leq \xi_{\max}(\mtx{A}; \mtx{B}; q) \leq \lambda_{\max}(\mtx{A}).
\end{equation}
The goal of our analysis is to understand how well $\xi_{\max}(\mtx{A}; \mtx{B}; q)$
\emph{approximates} the maximum eigenvalue $\lambda_{\max}(\mtx{A})$ as a function
of the block size $\ell$ and the depth $q$ of the Krylov space.

Provided that $\xi_{\max} \approx \lambda_{\max}$, any vector $\vct{v}_{\max}$
that maximizes the Rayleigh quotient in~\eqref{eqn:eigenvalue-estimate}
has a component in the invariant subspace of $\mtx{A}$ spanned by the large eigenvalues.
More precisely,
$$
\left( \frac{\norm{ \mtx{P}_{> \lambda} \vct{v}_{\max} }}{\norm{\vct{v}_{\max}}} \right)^2
	\geq \frac{\xi_{\max} - \lambda}{ \lambda_{\max} - \lambda}
	\quad\text{for all $\lambda < \xi_{\max}$.}
$$
We have written $\norm{\cdot}$ for the $\ell_2$ norm and $\mtx{P}_{> \lambda}$
for the spectral projector onto the invariant subspace of $\mtx{A}$ associated with the eigenvalues that
strictly exceed $\lambda$.

\subsubsection{Invariance Properties of the Eigenvalue Estimate}

The eigenvalue estimate $\xi_{\max}(\mtx{A}; \mtx{B}; q)$ inherits
some invariance properties from the block Krylov subspace.  These
facts can help us develop effective implementations of the algorithm
and to analyze their performance.

\begin{itemize}
\item	For fixed depth $q$, the estimate only depends on the range of the test matrix $\mtx{B}$:
\begin{equation} \label{eqn:range-invar}
\xi_{\max}(\mtx{A}; \mtx{BT}) = \xi_{\max}(\mtx{A}; \mtx{B})
	\quad\text{for all nonsingular $\mtx{T} \in \R^{\ell \times \ell}$.}
\end{equation}

\item	For fixed depth $q$, the estimate does not depend on the orientation of $\mtx{A}$ and $\mtx{B}$
in the sense that
\begin{equation} \label{eqn:rotation-invar}
\xi_{\max}(\mtx{UAU}^*; \mtx{UB}) = \xi_{\max}(\mtx{A}; \mtx{B}) 
	\quad\text{for all orthogonal $\mtx{U} \in \R^{n \times n}$.}
\end{equation}
\item	For fixed depth $q$, the estimate covaries with \emph{increasing} ``affine'' transformations of $\mtx{A}$:
\begin{equation} \label{eqn:affine-invar}
\xi_{\max}(\alpha \mtx{A} + \beta \Id; \mtx{B})
	= \alpha \xi_{\max}(\mtx{A}; \mtx{B}) + \beta
	\quad\text{for all $\alpha \geq 0$ and $\beta \in \R$.}
\end{equation}
\end{itemize}

\noindent
These results all follow immediately from the invariance properties of the Krylov
subspace (Section~\ref{sec:krylov-invar}) and the definition~\eqref{eqn:eigenvalue-estimate}
of the eigenvalue estimate.  See~\cite{KW92:Estimating-Largest} for the simple case $\ell = 1$.

\subsubsection{A Random Test Matrix}

To ensure that we can estimate the maximum eigenvalue of
an \emph{arbitrary} symmetric input matrix $\mtx{A}$, we draw the test
matrix $\mtx{B}$ that generates the Krylov subspace \emph{at random}.

How should we select the distribution?
Observe that the eigenvalue estimate $\xi_{\max}(\mtx{A}; \mtx{B}; q)$
only depends on the range of the test matrix $\mtx{B}$
because of the property~\eqref{eqn:range-invar}.
Furthermore, the property~\eqref{eqn:rotation-invar} shows that the
eigenvalue estimate is invariant under rotations.  Therefore,
we can choose any random test matrix whose range has a rotationally
invariant distribution.  This idea is extracted from~\cite{KW92:Estimating-Largest};
see Section~\ref{app:rot-invar} for the justification.

We will consider a standard normal test matrix $\mtx{\Omega} \in \R^{n \times \ell}$.
That is, the entries of $\mtx{\Omega}$ are statistically independent Gaussian random variables,
each with mean zero and variance one.  It is well known that the range
of this random matrix has a rotationally invariant distribution.
The goal of this paper is to study the behavior of the random eigenvalue estimate
$\xi_{\max}(\mtx{A}; \mtx{\Omega}; q)$.

\begin{remark}[Other Test Matrices]
Our analysis and detailed results depend heavily on the choice
of a standard normal test matrix $\mtx{\Omega}$.
In practice, we can achieve similar
empirical behavior from test matrices with ``structured'' distributions
that require less storage or that have fast matrix--vector multiplies.
For randomized Krylov subspace methods, the computational benefit of using
a structured test matrix is limited because we need to perform repeated
multiplications with the input matrix $\mtx{A}$ to generate the Krylov matrix;
cf.~\eqref{eqn:krylov-matrix}.  See~\cite[Secs.~4.6 and 7.4]{HMT11:Finding-Structure}
or~\cite[Secs.~9 and~11.5]{MT20:Randomized-Numerical}
for some discussion about other random test matrices.
\end{remark}

\subsubsection{Implementation}
\label{sec:implementation}

For completeness, we describe the simplest stable implementation of a block Krylov
method for computing the largest eigenvalue of a symmetric matrix.
See Algorithm~\ref{alg:maxeig} for pseudocode; %
this approach is adapted from~\cite{RST09:Randomized-Algorithm,HMST11:Algorithm-Principal}.
Here is a summary of the costs: %

\begin{itemize}
\item 	A total of $q$ matrix--matrix multiplies between $\mtx{A}$ and an $n \times \ell$ matrix,
plus another multiplication between $\mtx{A}$ and an $n \times (q+1)\ell$ matrix.
The arithmetic cost depends on whether the matrix $\mtx{A}$ supports a fast multiplication
operation.  For example, the algorithm is more efficient when $\mtx{A}$ is sparse.

\item	Repeated orthogonalization of $(q+1) \ell$ vectors of length $n$
with total cost $\mathcal{O}(q^2 \ell^2 n)$ operations.

\item	Solution of a (block-tridiagonal) maximum eigenvalue problem at a cost of $\mathcal{O}(q^2 \ell^2)$ operations.

\item	Maintenance of the matrix $\mtx{S}_q$, which requires $\mathcal{O}(q\ell n)$ units of storage.
\end{itemize}

\noindent
More refined algorithms can reduce the resource usage markedly;
see~\cite[Sec.~11.7]{MT20:Randomized-Numerical}
or~\cite{TW21:Randomized-Block} for discussion.

As noted, Krylov subspace methods are particularly valuable when we have an efficient
procedure for computing matrix--vector products with $\mtx{A}$.
On many contemporary computer architectures, the cost of performing
a product with several vectors is similar to the cost of a product
with a single vector.  In this setting, we can take advantage of
the improved behavior of a block method almost for free.

We refer to the books~\cite{Par98:Symmetric-Eigenvalue,BDD+00:Templates-Solution,
Saa11:Numerical-Methods,GVL13:Matrix-Computations} and the
paper~\cite{HMST11:Algorithm-Principal} for more discussion
and references.

\begin{algorithm}[tb]
  \caption{A block Krylov method for computing the largest eigenvalue of a symmetric matrix}
    \label{alg:maxeig}
  \begin{algorithmic}[1]
    \Require{Symmetric $n \times n$ matrix $\mtx{A}$; block size $\ell$; depth $q$}
    \Ensure{Estimate $\xi$ of largest eigenvalue}
\vspace{0.5pc}
    \Function{BlockKrylovMaxEig}{$\mtx{A}$, $\ell$, $q$}
    	\State	$[\mtx{B}_0,\sim] \gets \texttt{qr\_econ}( \texttt{randn}(\texttt{size}(\mtx{A}, 1), \ell) )$
			\Comment{Draw $n \times \ell$ random orthonormal matrix}
		\State	$\mtx{S}_0 = \mtx{B}_0$
			\Comment{Orthonormal basis for Krylov subspace}
		\For{$t \gets 1, 2, 3, \dots, q$}
			\State	$\mtx{B}_t \gets \mtx{A} \mtx{B}_{t-1}$
				\Comment{Form next block of Krylov matrix by multiplication}
			\State  $\mtx{B}_t \gets \mtx{B}_t - \mtx{S}_{t-1} (\mtx{S}_{t-1}^* \mtx{B}_t)$
				\Comment{Orthogonalize against subspace by double GS}
			\State	$\mtx{B}_t \gets \mtx{B}_t - \mtx{S}_{t-1} (\mtx{S}_{t-1}^* \mtx{B}_t)$
				\Comment{\textit{Sic!}}
			\State  $[\mtx{B}_t, \sim] \gets \texttt{qr\_econ}(\mtx{B}_t)$
				\Comment{Orthogonalize what's left}
			\State	$\mtx{S}_t \gets [\mtx{S}_{t-1}, \mtx{B}_t]$
				\Comment{Extend basis with new block}
		\EndFor
		\State	$\mtx{H} \gets \mtx{S}_q^* (\mtx{A} \mtx{S}_q)$
			\Comment{Form block-tridiagonal Rayleigh quotient matrix}
		\State	$[\xi, \vct{v}] \gets \texttt{maxeig}(\mtx{H})$
			\Comment{Find maximum eigenpair}
    \EndFunction 
	\vspace{0.25pc}
\end{algorithmic}
\end{algorithm}

\subsection{The Role of the Spectrum}

Owing to invariance, the theoretical behavior of the eigenvalue
estimate $\xi_{\max}(\mtx{A}; \mtx{\Omega}; q)$ depends only
on the spectrum of the input matrix $\mtx{A}$.  In this section,
we develop this idea further and introduce some spectral features
that affect the performance of the eigenvalue estimate.

\begin{mywarning}[Numerical Behavior]
The numerical performance of a (block) Krylov method can be very
complicated; for example, see~\cite[Chap.~13]{Par98:Symmetric-Eigenvalue}
or~\cite{Meu06:Lanczos-Conjugate}.
The current paper does not address numerics. %
\end{mywarning}

\subsubsection{Invariance Properties of the Random Eigenvalue Estimate}

The random estimate $\xi_{\max}(\mtx{A}; \mtx{\Omega}; q)$ of the maximum
eigenvalue has several invariance properties that allow us to simplify the analysis.

First, the rotation invariance~\eqref{eqn:rotation-invar} of the eigenvalue
estimate and the rotational invariance of the range of $\mtx{\Omega}$ imply
that
\begin{equation} \label{eqn:input-diag}
\xi_{\max}(\mtx{A}; \mtx{\Omega}) \sim \xi_{\max}(\mtx{\Lambda}; \mtx{\Omega})
\quad\text{where $\mtx{A} = \mtx{U\Lambda U}^*$ is an eigenvalue factorization.}
\end{equation}
The symbol $\sim$ signifies equality of distribution for two random variables.
In other words, the maximum eigenvalue estimate depends only on the eigenvalues
of the input matrix---but not the eigenvectors.

Second, owing to the affine covariance property~\eqref{eqn:affine-invar},
the eigenvalue estimate $\xi_{\max}(\mtx{A}; \mtx{\Omega}; q)$
only depends on the ``shape'' of the spectrum of $\mtx{A}$,
but not its location or scale.  As a consequence, we must assess the
behavior of the eigenvalue estimate in terms of spectral features
that are affine invariant.

\subsubsection{Spectral Features of the Input Matrix}

To express the results of our analysis,
we introduce terminology for some spectral features of the
symmetric matrix $\mtx{A} \in \R^{n \times n}$.
First, let us instate compact notation for the eigenvalues
of $\mtx{A}$:
$$
a_i := \lambda_i(\mtx{A})
\quad\text{for $i = 1, \dots, n$,}
\quad\text{and}\quad
a_{\max} := a_1 \geq a_2 \geq \dots \geq a_n =: a_{\min}.
$$
The map $\lambda_i(\cdot)$ returns the $i$th largest eigenvalue of a symmetric matrix.

Next, we define some functions of the eigenvalue spectrum:

\begin{itemize}
\item	The \emph{spectral range} $\rho$ of the input matrix is the distance between the extreme eigenvalues.  That is, $\rho := a_{\max} - a_{\min}$. %

\item	The \emph{spectral gap} $\gamma$ is the relative difference between the maximum eigenvalue
and the next distinct eigenvalue:
\begin{equation} \label{eqn:spectral-gap}
\gamma := \frac{a_{\max} - a_{m+1}}{a_{\max} - a_{\min}}
\quad\text{where}\quad
a_{\max} = a_{m} > a_{m+1}.
\end{equation}
If $\mtx{A}$ is a multiple of the identity, then $\gamma = 0$.  Note that $\gamma \in [0, 1]$.

\item	Let $\nu$ be a nonnegative number.  The \emph{$\nu$-stable rank} is a continuous measure
of the ``dimension'' of the range of $\mtx{A} - a_{\min} \Id$ that reflects how quickly the spectrum
decays.  It is defined as
\begin{equation} \label{eqn:srank}
\srank(\nu) := \sum\limits_{i=1}^n \left( \frac{a_i - a_{\min}}{a_{\max} - a_{\min}} \right)^{2\nu}.
\end{equation}
If $\mtx{A}$ is a multiple of the identity, we define $\srank(\nu) = 0$.
Otherwise, $1 \leq \srank(\nu) \leq \rank(\mtx{A} - a_{\min} \Id) \leq n - 1$.
When the eigenvalues of $\mtx{A}$ decay at a polynomial rate, the stable rank can be much smaller
than the rank for an appropriate choice of $\nu$.
See Section~\ref{sec:numerics-setup} for some illustrations.

\item	Let $\xi$ be any estimate for the largest eigenvalue $a_{\max}$ of the input matrix $\mtx{A}$.
We measure the error in the estimate relative to the spectral range:
\begin{equation} \label{eqn:error-def}
\mathrm{err}(\xi) := \frac{a_{\max} - \xi}{a_{\max} - a_{\min}}.
\end{equation}
The relative error in the Krylov estimate $\xi = \xi_{\max}(\mtx{A}; \mtx{B}; q)$
falls in the interval $[0, 1]$ because of~\eqref{eqn:underestimate}. 
\end{itemize}

\noindent
The spectral gap, the stable rank, and the error measure are all
invariant under increasing affine transformations of the spectrum of $\mtx{A}$.
We suppress the dependence of these quantities on the input matrix $\mtx{A}$,
unless emphasis is required.

\begin{remark}[History]
Concepts related to the $\nu$-stable rank originally appeared in the
analysis literature~\cite{BT87:Invertibility-Large,BT91:Problem-Kadison},
and they now play an important role in randomized linear algebra~\cite[Sec.~2.6]{MT20:Randomized-Numerical}.
\end{remark}

\subsubsection{An Example}
\label{sec:laplacian}

Let us illustrate these concepts with a canonical example.
A common task in computational physics is to compute the \textit{minimum}
eigenvalue of (the discretization of) an elliptic operator.

Consider the standard $n$-point second-order finite-difference discretization $\mtx{L} \in \R^{n \times n}$
of the one-dimensional Laplacian on $[0, 1]$ with homogeneous boundary conditions.
Writing $h = 1/(n+1)$, the eigenvalues of the matrix $\mtx{L}$ are
$$
\lambda_{j}(\mtx{L}) = \frac{2}{h^2} \left[ 1 + \cos( \pi j h ) \right]
\quad\text{for $j = 1, \dots, n$.}
$$
The spectral gap between the smallest and second smallest eigenvalue
is proportional to $h^2$. %
The spectrum has essentially no decay, which is visible in the fact that
$\srank(-\mtx{L}; 1)$ is proportional to $1/h$.
The tiny spectral gap and large stable rank suggest that the maximum eigenvalue problem for $-\mtx{L}$
will be challenging.

A natural remedy is to attempt to compute the maximum eigenvalue
of the inverse $\mtx{L}^{-1}$.  Independent of $n$,
the spectral gap between the largest and second largest
eigenvalue satisfies $\gamma( \mtx{L}^{-1} ) \approx 0.75$.
The stable rank $\srank(\mtx{L}^{-1}; 1) \approx 1.1$.
The large spectral gap and small stable rank suggest that the maximum eigenvalue
problem for $\mtx{L}^{-1}$ will be quite easy, regardless of the dimension.

\subsection{Matrices with Few Distinct Eigenvalues}

Before continuing, we must address an important special case.
When the input matrix has few distinct eigenvalues, the block Krylov method
computes the maximum eigenvalue of the matrix perfectly.

\begin{proposition}[Randomized Block Krylov: Matrices with Few Eigenvalues] \label{prop:few-eigs}
Let $\mtx{A} \in \R^{n \times n}$ be a symmetric matrix.  Fix 
the block size $\ell \geq 1$ and the depth $q \geq 0$
of the block Krylov subspace. %
Draw a standard normal matrix $\mtx{\Omega} \in \R^{n \times \ell}$.
If $\mtx{A}$ has $q + 1$ or fewer distinct eigenvalues, then $\mathrm{err}(\xi_{\max}(\mtx{A}; \mtx{\Omega};q)) = 0$ with probability one.
\end{proposition}

\noindent
This type of result is well known (e.g., see~\cite{KW92:Estimating-Largest}),
but we include a short proof in Section~\ref{sec:error-formula}.

\subsection{Matrices without a Spectral Gap}
\label{sec:gapfree}

Our first result gives probabilistic bounds for the maximum eigenvalue
estimate $\xi_{\max}(\mtx{A}; \mtx{\Omega}; q)$ without any additional
assumptions.  In particular, it does not require a lower bound on the
spectral gap $\gamma$.

\begin{theorem}[Randomized Block Krylov: Maximum Eigenvalue Estimate] \label{thm:gapfree}
Instate the following hypotheses.

\begin{itemize}
\item	Let $\mtx{A} \in \R^{n \times n}$ be a symmetric input matrix.

\item	Draw a standard normal test matrix $\mtx{\Omega} \in \R^{n \times \ell}$ with block size $\ell$.

\item	Fix the depth parameter $q \geq 0$, and let $q = q_1 + q_2$ be an arbitrary nonnegative integer partition.
\end{itemize}

\noindent
We have the following probability bounds for the estimate
$\xi_{\max}(\mtx{A}; \mtx{\Omega}; q)$, defined in~\eqref{eqn:eigenvalue-estimate},
of the maximum eigenvalue of the input matrix.

\begin{enumerate}
\item	The relative error~\eqref{eqn:error-def} in the eigenvalue estimate satisfies the probability bound
\begin{equation} \label{eqn:intro-prob-gapfree}
\Prob{ \mathrm{err}(\xi_{\max}(\mtx{A}; \mtx{\Omega};q)) \geq \eps }
	\leq 1 \wedge \sqrt{2} \left[ 8 \srank(q_1) \cdot \econst^{-2(2q_2+1) \sqrt{\eps}} \right]^{\ell/2}
	\quad\text{for $\eps \in [0, 1]$.}
\end{equation}

\item	The expectation of the relative error satisfies
\begin{equation} \label{eqn:intro-expect-gapfree}
\Expect \mathrm{err}(\xi_{\max}(\mtx{A}; \mtx{\Omega}; q))
	\leq 1 \wedge \left[ \frac{2.70 \ell^{-1} + \log(8 \srank(q_1))}{2(2q_2+1)} \right]^2.
\end{equation}
\end{enumerate}
The symbol $\wedge$ denotes the minimum,
and the stable rank, $\srank(\cdot)$, is defined in~\eqref{eqn:srank}.
All logarithms are natural.
\end{theorem}

\noindent
The proof of~\eqref{eqn:intro-prob-gapfree} appears in Section~\ref{sec:prob-bounds},
and the proof of~\eqref{eqn:intro-expect-gapfree} appears in Section~\ref{sec:expect-bound-gapfree}.
The experiments in Section~\ref{sec:numerics} support the analysis.

Let us take a moment to explain the content of this result.
We begin with a discussion about the role of the second depth parameter $q_2$,
and then we explain the role of the first depth parameter $q_1$.
We emphasize that the user does not choose a partition $q = q_1 + q_2$;
the bounds are valid for all partitions.

For now, fix $q_1$.  The key message of Theorem~\ref{thm:gapfree}
is that the relative error satisfies %
$$
\Expect \mathrm{err}(\xi_{\max}(\mtx{A}; \mtx{\Omega};q_1 + q_2)) \leq \eps
$$
once the depth parameter $q_2$ exceeds 
$$
q_{2}(\eps) := - \frac{1}{2} + \frac{2.70 \ell^{-1} + \log( 8 \srank(q_1) )}{4\sqrt{\eps}}.
$$
Once the depth $q_2$ attains this level, the probability of error drops off exponentially fast:
$$
q_2 = q_{2}(\eps) + k \eps^{-1/2}
	\quad\text{implies}\quad
	\Prob{\mathrm{err}(\xi_{\max}(\mtx{A}; \mtx{\Omega}; q_1 + q_2)) \geq \eps} \leq \econst^{-2k\ell}.
$$
In fact, we need $q_2 \geq q_2(1)$ just to ensure that the probability bound is nontrivial.

The most important aspect of this result is that the depth $q_{2}(\eps)$ scales
with $\eps^{-1/2}$, so it is possible to achieve moderate relative error
using a block Krylov space with limited depth.  In contrast, the randomized power
method~\cite{KW92:Estimating-Largest} requires the depth $q$ to be proportional to $\eps^{-1}$
to achieve a relative error of $\eps$.

The second thing to notice is that the depth $q_{2}(\eps)$ scales with $\log(\srank(q_1))$.
The stable rank is never larger than the ambient dimension $n$, but
it can be significantly smaller---even constant---when the spectrum
of the matrix has polynomial decay.

Here is another way to look at these facts.
As we increase the depth parameter $q$,
the block Krylov method exhibits a burn-in period whose length $q_1 + q_2(1)$ depends on
$\srank(q_1)$. %
While the depth $q_2 \leq q_2(1)$,
the algorithm does not make much progress in estimating the maximum eigenvalue.
Once the depth satisfies $q_2 \geq q_2(1)$, the expected
relative error decreases in proportion to $q_2^{-2}$.  In contrast, the power method~\cite{KW92:Estimating-Largest}
reduces the expected relative error in proportion to $q_2^{-1}$.

We can now appreciate the role of the first depth parameter $q_1$.
When the spectrum of the input matrix exhibits polynomial decay, $\srank(q_1)$ is \emph{constant}
for an appropriate value of $q_1$ that depends on the decay rate.
In this case, the analysis shows that the total burn-in period $q_1 + q_2(1)$
is just $\mathcal{O}(1)$ steps.  For example, when the $j$th eigenvalue decays like $1/j$,
this situation occurs. %

The block size $\ell$ may not play a significant role in determining
the average error.  But changing the block size has a large effect
on the probability of failure (i.e., the event that the relative error exceeds $\eps$).
For example, suppose that we increase the block size $\ell$ from one to three.
For each increment of $\eps^{-1/2}$ in the depth $q_2$,
the failure probability with block size $\ell = 3$ is a factor of $54\times$
smaller than the failure probability with block size $\ell = 1$!

\begin{remark}[Prior Work] \label{rem:prior-gapfree}
The simple case $\ell = 1$ in Theorem~\ref{thm:gapfree}
has been studied in the paper~\cite{KW92:Estimating-Largest}.
Our work introduces two major innovations.  First, we obtain
bounds in terms of the stable rank, which allows us
to mitigate the dimensional dependence that appears
in~\cite{KW92:Estimating-Largest}.  Second, we have
obtained precise results for larger block sizes $\ell$,
which indicate potential benefits of using block Krylov
methods.  Our proof strategy is motivated by the work
in~\cite{KW92:Estimating-Largest}, but we have been
able to streamline and extend their arguments by using
a more transparent random model for the test matrix.
\end{remark}

\subsection{Matrices with a Spectral Gap}
\label{sec:gap}

Our second result gives probabilistic bounds for the maximum eigenvalue estimate
$\xi_{\max}(\mtx{A}; \mtx{\Omega};q)$ when we have a lower bound for the spectral gap
$\gamma$ of the input matrix.

\begin{theorem}[Randomized Block Krylov: Maximum Eigenvalue Estimate with Spectral Gap] \label{thm:gap}
Instate the following hypotheses.

\begin{itemize}
\item	Let $\mtx{A} \in \R^{n \times n}$ be a symmetric input matrix with spectral gap $\gamma$, defined in~\eqref{eqn:spectral-gap}.

\item	Draw a standard normal test matrix $\mtx{\Omega} \in \R^{n \times \ell}$ with block size $\ell$.

\item	Fix the depth parameter $q \geq 0$, and let $q = q_1 + q_2$ be an arbitrary nonnegative integer partition.
\end{itemize}

\noindent
We have the following probability bounds for the estimate
$\xi_{\max}(\mtx{A}; \mtx{\Omega}; q)$, defined in~\eqref{eqn:eigenvalue-estimate},
of the maximum eigenvalue of the input matrix.

\begin{enumerate}
\item	The relative error~\eqref{eqn:error-def} in the eigenvalue estimate satisfies the probability bound
\begin{equation}  \label{eqn:intro-prob-gap}
\Prob{ \mathrm{err}(\xi_{\max}(\mtx{A}; \mtx{\Omega}; q)) \geq \eps }
	\leq 1 \wedge \sqrt{2} \left[ \frac{8\srank(q_1)}{\eps} \cdot \econst^{-4q_2 \sqrt{\gamma}} \right]^{\ell/2}
	\quad\text{for $\eps \in (0,1]$.} %
\end{equation}

\item	Abbreviate $F := 4 \srank(q_1) \econst^{-4q_2 \sqrt{\gamma}}$.
The expectation of the relative error satisfies %
\begin{align} %
\Expect \mathrm{err}(\xi_{\max}(\mtx{A}; \mtx{\Omega}))
	&\leq \frac{F}{(\ell - 2) + F} %
	&& (\ell \geq 3); \label{eqn:gap-l>3} \\
\Expect \mathrm{err}(\xi_{\max}(\mtx{A}; \mtx{\Omega}))
	&\leq (F/2) \log(1 + 2/F)
	&& (\ell = 2); \label{eqn:gap-l=2} \\
\Expect \mathrm{err}(\xi_{\max}(\mtx{A}; \mtx{\Omega}))
	&\leq 1 \wedge \sqrt{2 \pi F}  %
	&& (\ell = 1). \label{eqn:gap-l=1}
\end{align}
\end{enumerate}
The symbol $\wedge$ denotes the minimum,
and the stable rank is defined in~\eqref{eqn:srank}.
\end{theorem}

\noindent
The proof of~\eqref{eqn:intro-prob-gap} appears in Sections~\ref{sec:prob-bounds};
the proof of~\eqref{eqn:gap-l>3}, \eqref{eqn:gap-l=2}, and \eqref{eqn:gap-l=1}
appears in Section~\ref{sec:expect-bound-gap}.
The experiments in Section~\ref{sec:numerics} bear out these predictions.

Let us give a verbal summary of what this result means.
First of all, we anticipate that Theorem~\ref{thm:gap}
will give better bounds than Theorem~\ref{thm:gapfree}
when the spectral gap $\gamma$ exceeds the target level $\eps$ for the error.
But both results are valid for all choices of $\gamma$ and $\eps$.

Now, fix the first depth parameter $q_1$.
One implication of the spectral gap result is that the relative error satisfies
$$
\Prob{ \mathrm{err}(\xi_{\max}(\mtx{A}; \mtx{\Omega};q)) \leq \eps } \approx 1
$$
when the second depth parameter $q_2$ exceeds
$$
q_{2}(\eps; \gamma) := \frac{  0.70 \ell^{-1} + \log(8 \eps^{-1} \srank(q_1)) }{ 4 \sqrt{\gamma} }.
$$
In this case, the depth $q_2$ only scales with $\log(1/\eps)$, so the block Krylov method
can achieve very small relative error for a matrix with a spectral gap.
As before, the depth $q_2$ also scales with
$\log(\srank(q_1))$, so the dimensional dependence is weak---or even nonexistent
if the spectrum has polynomial decay and $q_1$ is sufficiently large.

When the depth $q_2 \geq q_{2}(\eps; \gamma)$, the error probability drops off quickly:
$$
q_2 = q_2(\eps; \gamma) + k \gamma^{-1/2}
\quad\text{implies}\quad
\Prob{ \mathrm{err}(\xi_{\max}(\mtx{A}; \mtx{\Omega})) \geq \eps }
	\leq \econst^{-2k \ell}.
$$
This bound indicates that $\gamma^{-1/2}$ is the scale on which the depth $q_2$
needs to increase to reduce the failure probability by a constant multiple
(which depends on the block size).

We discover a new phenomenon when we examine the expectation of the error.
On average, to achieve a relative error of $\eps$, it is sufficient that
the depth $q_2 \geq q_2'(\eps; \gamma)$, where
$$
\begin{aligned}
\ell \geq 3 &:
\quad q_2'(\eps;\gamma) := \frac{\log(4 \srank(q_1)) + \log(1/\eps) - \log(\ell - 2)}{4 \sqrt{\gamma}}; \\
\ell = 2 &:
\quad q_2'(\eps;\gamma) := \frac{\log(4 \srank(q_1)) + \log(1/\eps) + \log\log(1/\eps)}{4 \sqrt{\gamma}}
	& \text{for all $\eps \leq \mathrm{const}$;} \\
\ell = 1 &:
\quad q_2'(\eps;\gamma) := \frac{\log(4 \srank(q_1)) + 2 \log(1/\eps) + \log(2\pi)}{4 \sqrt{\gamma}}.
\end{aligned}
$$
In other words, the depth $q_2$ of the block Krylov space needs to be about
$\log(4 \srank(q_1)) / (4 \sqrt{\gamma})$ before we obtain an average
relative error less than one; we can reduce this requirement slightly
by increasing the block size $\ell$.  But once the depth $q_2$ reaches
this level, Theorem~\ref{thm:gap} suggests that the block Krylov method
with $\ell \geq 2$ reduces the average error \emph{twice as fast}
as the block Krylov method with $\ell = 1$.

\begin{remark}[Prior Work]
The simple case $\ell = 1$ in Theorem~\ref{thm:gap}
has been studied in the paper~\cite{KW92:Estimating-Largest}.
See Remark~\ref{rem:prior-gapfree} for a discussion of
the improvements we have achieved.
\end{remark}

\subsection{Estimating Minimum Eigenvalues}

We can also use Krylov subspace methods to obtain an estimate $\xi_{\min}(\mtx{A}; \mtx{B}; q)$
for the \emph{minimum} eigenvalue of a symmetric matrix $\mtx{A}$.  Conceptually, the simplest
way to accomplish this task is to apply the Krylov subspace method to the negation $- \mtx{A}$.
The minimum eigenvalue estimate takes the form
$$
\xi_{\min}(\mtx{A}; \mtx{B}; q) := - \xi_{\max}(-\mtx{A}; \mtx{B}; q).
$$
Owing to~\eqref{eqn:underestimate}, this estimate is never smaller than $\lambda_{\min}(\mtx{A})$.

It is straightforward to adapt Theorems~\ref{thm:gapfree} to obtain
bounds for the minimum eigenvalue estimate with a random test matrix $\mtx{\Omega}$.
In particular, we always have the bound
$$
\Expect \left[ \frac{\xi_{\min}(\mtx{A}; \mtx{\Omega}; q) - a_{\min}}{a_{\max} - a_{\min}} \right]
	\leq \left[ \frac{2.70 \ell^{-1} + \log(8 \srank(-\mtx{A}; q_1))}{2(2q_2+1)} \right]^2. %
$$
In this context, the stable rank takes the form
$$
\srank(-\mtx{A}; \nu) = \sum\limits_{i = 1}^n \left(\frac{a_{\max} - a_i}{a_{\max} - a_{\min}} \right)^{2\nu}.
$$
See Section~\ref{sec:gapfree} for discussion of this type of bound.

We can also use Theorem~\ref{thm:gap} to obtain results in terms of the spectral gap.
The spectral gap $\gamma(-\mtx{A})$ is the magnitude of the difference between the
smallest two eigenvalues of $\mtx{A}$, relative to the spectral range.
For example, when the block size $\ell = 3$, we have the bound
$$
\ell = 3 : \quad\quad
\Expect \left[ \frac{\xi_{\min}(\mtx{A}; \mtx{\Omega};q) - a_{\min}}{a_{\max} - a_{\min}} \right]
	\leq 4 \srank(-\mtx{A}; q_1) \cdot \econst^{-4 q_2 \sqrt{\gamma(-\mtx{A})}}. %
$$
See Section~\ref{sec:gap} for discussion of this type of bound.

\begin{remark}[Inverse Iterations]
Suppose that we have a routine for applying the \emph{inverse}
of a positive-semidefinite matrix $\mtx{A}$ to a vector.
In this case, we can apply the Krylov method to the inverted
matrix $\mtx{A}^{-1}$ to estimate the minimum eigenvalue.
This approach is often more powerful than applying
the Krylov method to $-\mtx{A}$.
For example, inverse iteration is commonly used
for discretized elliptic operators; the discussion
in Section~\ref{sec:laplacian} supports this approach.
\end{remark}

\subsection{Estimating Singular Values}

We now arrive at the problem of estimating the spectral norm of a general matrix
$\mtx{C} \in \R^{n \times m}$ using Krylov subspace methods.  Assuming that $n \leq m$,
we can apply the block Krylov method to the square $\mtx{CC}^*$.
This yields an estimate $\xi_{\max}(\mtx{CC}^*; \mtx{B}; q)$ for the \emph{square}
of the spectral norm of $\mtx{C}$.  If $n > m$, we can just as well work with the
other square $\mtx{C}^* \mtx{C}$.

Theorems~\ref{thm:gapfree} and~\ref{thm:gap} immediately yield
error bounds for the random test matrix $\mtx{\Omega}$.
In particular, we always have the bound
$$
\Expect \left[ \frac{\norm{\mtx{C}}^2 - \xi_{\max}(\mtx{CC}^*; \mtx{\Omega};q)}{\norm{\mtx{C}}^2} \right]
	\leq \left[ \frac{2.70 \ell^{-1} + \log(8 \srank(\mtx{CC}^*;q_1))}{2(2q_2+1)} \right]^2. %
$$
We also have a bound in terms of the spectral gap $\gamma(\mtx{CC}^*)$, which is
the difference between the squares of the largest two distinct singular values,
relative to the spectral range.
For block size $\ell = 3$, we have
$$
\Expect \left[ \frac{\norm{\mtx{C}}^2 - \xi_{\max}(\mtx{CC}^*; \mtx{\Omega})}{\norm{\mtx{C}}^2} \right]
	\leq 4 \srank(\mtx{CC}^*;q_1) \cdot \econst^{-4 q_2 \sqrt{\gamma(\mtx{CC}^*)}}.
$$
In this case, it is natural to bound the stable rank as
$$
\srank(\mtx{CC}^*; 0) = m \wedge n
\quad\text{and}\quad
\srank(\mtx{CC}^*; \nu) \leq \left(\frac{\norm{\mtx{C}}_{4\nu}}{\norm{\mtx{C}}}\right)^{4\nu}
\quad\text{for $\nu \geq 1$.}
$$
We have written $\norm{\cdot}$ for the spectral norm and $\norm{\cdot}_{p}$ for the Schatten $p$-norm
for $p \geq 1$.

\begin{remark}[Other Approaches]
It is also possible to work with ``odd'' Krylov subspaces
$K_q(\mtx{CC}^*; \mtx{CB})$ or $K_q(\mtx{C}^*\mtx{C}; \mtx{C}^* \mtx{B})$,
but the analysis requires some modifications.
\end{remark}

\begin{remark}[Minimum Singular Values]
The quantity $\xi_{\min}(\mtx{CC}^*; \mtx{B})$ gives an estimate for
the $m$th largest squared singular value of $\mtx{C}$.  When $m \leq n$,
this is the smallest singular value, which may be zero.  It is straightforward to
derive results for the estimate using the principles outlined above.
We omit the details.
\end{remark}

\subsection{Extensions}

With minor modifications, the analysis in this paper can be extended
to cover some related situations.  First, when the maximum eigenvalue has multiplicity greater
than one, the (block) Krylov method converges more quickly.
Second, we can extend the algorithm and the analysis
to the problem of estimating the largest eigenvalue of an Hermitian
matrix (using a complex standard normal test matrix).  Third,
we can analyze the behavior of the randomized (block) power method.
For brevity, we omit the details.

\section{Numerical Experiments}
\label{sec:numerics}

Krylov subspace methods exhibit complicated behavior because
they are implicitly solving a polynomial optimization problem.
Therefore, we do not expect a reductive theoretical analysis to
capture all the nuances of their behavior.  Nevertheless,
by carefully choosing input matrices, we can witness
the phenomena that the theoretical analysis predicts.

\subsection{Experimental Setup}
\label{sec:numerics-setup}

We implemented the randomized block Krylov method in \textsc{Matlab} 2019a.
The code uses full orthogonalization, as described in Algorithm~\ref{alg:maxeig},
and the maximum eigenvalue of the Rayleigh matrix is computed
with the command \texttt{eig}.

We consider several types of input matrices for which accurate
maximum eigenvalue estimates are particularly difficult.
The randomized block Krylov method is rotationally invariant and affine invariant,
so there is no (theoretical) loss in working with a diagonal matrix.
Of course, the Krylov method does not exploit the fact that the
input matrix is diagonal.

Recall that an $n \times n$ matrix
$\mtx{W} \in \R^{n \times n}$ from the Gaussian Orthogonal Ensemble (GOE)
is obtained by extracting the symmetric part of a Gaussian matrix:
$$
\mtx{W} = \half (\mtx{G} + \mtx{G}^*) \in \R^{n \times n}
\quad\text{where $\mtx{G} \in \R^{n \times n}$ has independent standard normal entries.}
$$
The large eigenvalues of a GOE matrix cluster together,
and the spectral gap becomes increasingly small as the
dimension increases.  The spectrum has essentially no decay.

\begin{itemize}
\item	To study the behavior of the Krylov method
\textit{without} a spectral gap, we draw a realization of a GOE matrix,
diagonalize it using \texttt{eig}, and make an affine transformation
so that its extreme eigenvalues are $0$ and $1$.  Abusing terminology,
we also refer to this diagonal input matrix as a GOE matrix.

\item	To study the behavior of the Krylov method
\textit{with} a spectral gap $\gamma$, we take the diagonalized GOE matrix
and increase the largest eigenvalue until the gap is $\gamma$.
We refer to this model as a \textit{gapped GOE matrix}.

\item	To understand how tail decay affects the convergence of the
Krylov method, we consider \textit{gapped power law} matrices.
For a dimension $n$, power $p > 0$, and gap $\gamma \in [0, 1)$,
the matrix is diagonal with entries
$$
a_1 = 1 + \frac{\gamma}{1-\gamma}
\quad\text{and}\quad
a_i = (i-1)^{-1/p}
\quad\text{for $i = 2, \dots, n$.}
$$
\end{itemize}

See Figure~\ref{fig:srank} for a depiction of the stable rank function
for these types of input matrices.
For the random matrix models, we draw a single realization of the input
matrix and then fix it once and for all.  The variability in the experiments
derives from the random draw of the test matrix.

\begin{figure}
\begin{subfigure}{0.49\textwidth}
\includegraphics[width=\columnwidth]{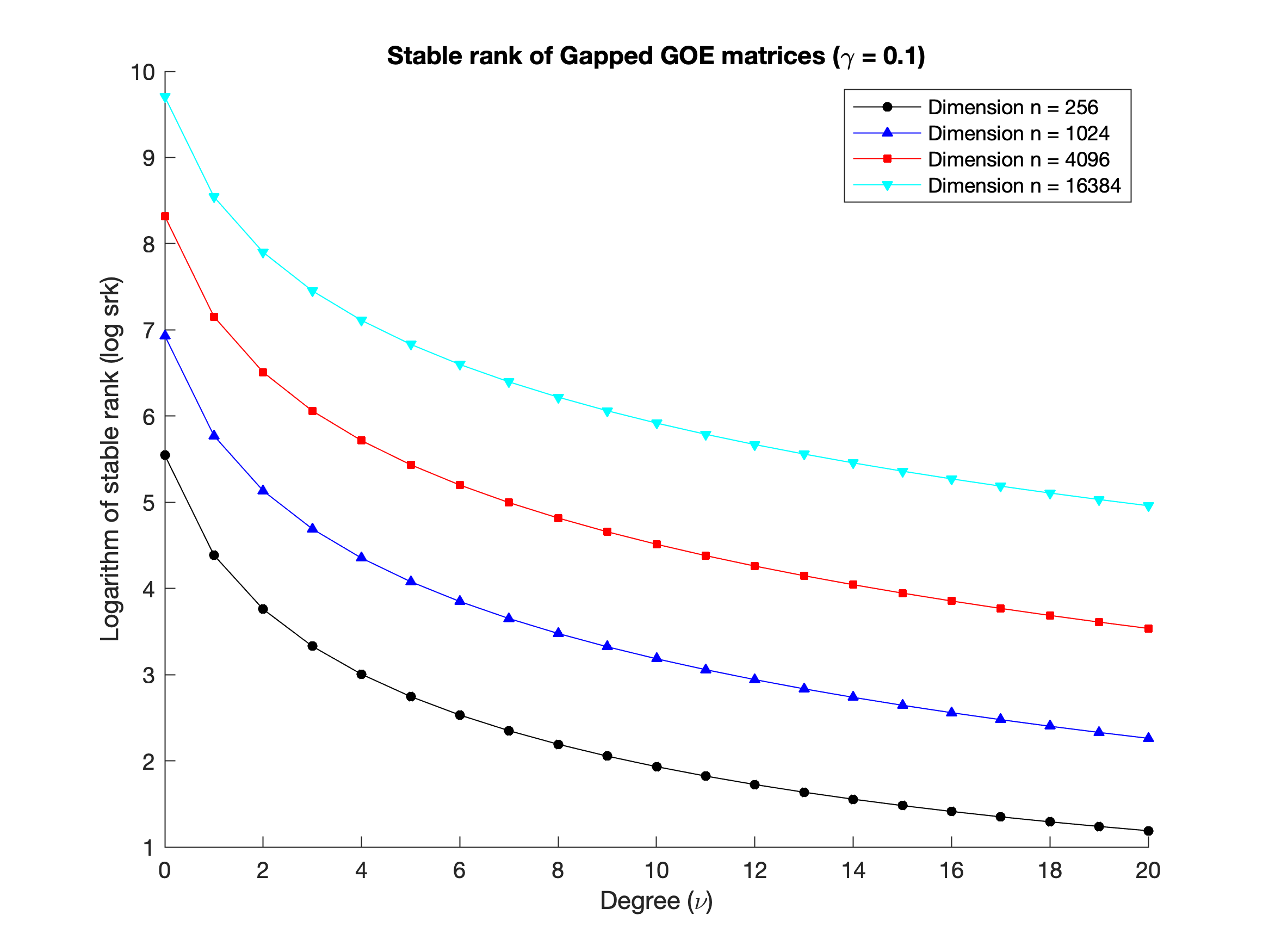}
\caption{Gapped GOE matrices} \label{fig:goe-srk}
\end{subfigure}
\begin{subfigure}{0.49\textwidth}
\includegraphics[width=\columnwidth]{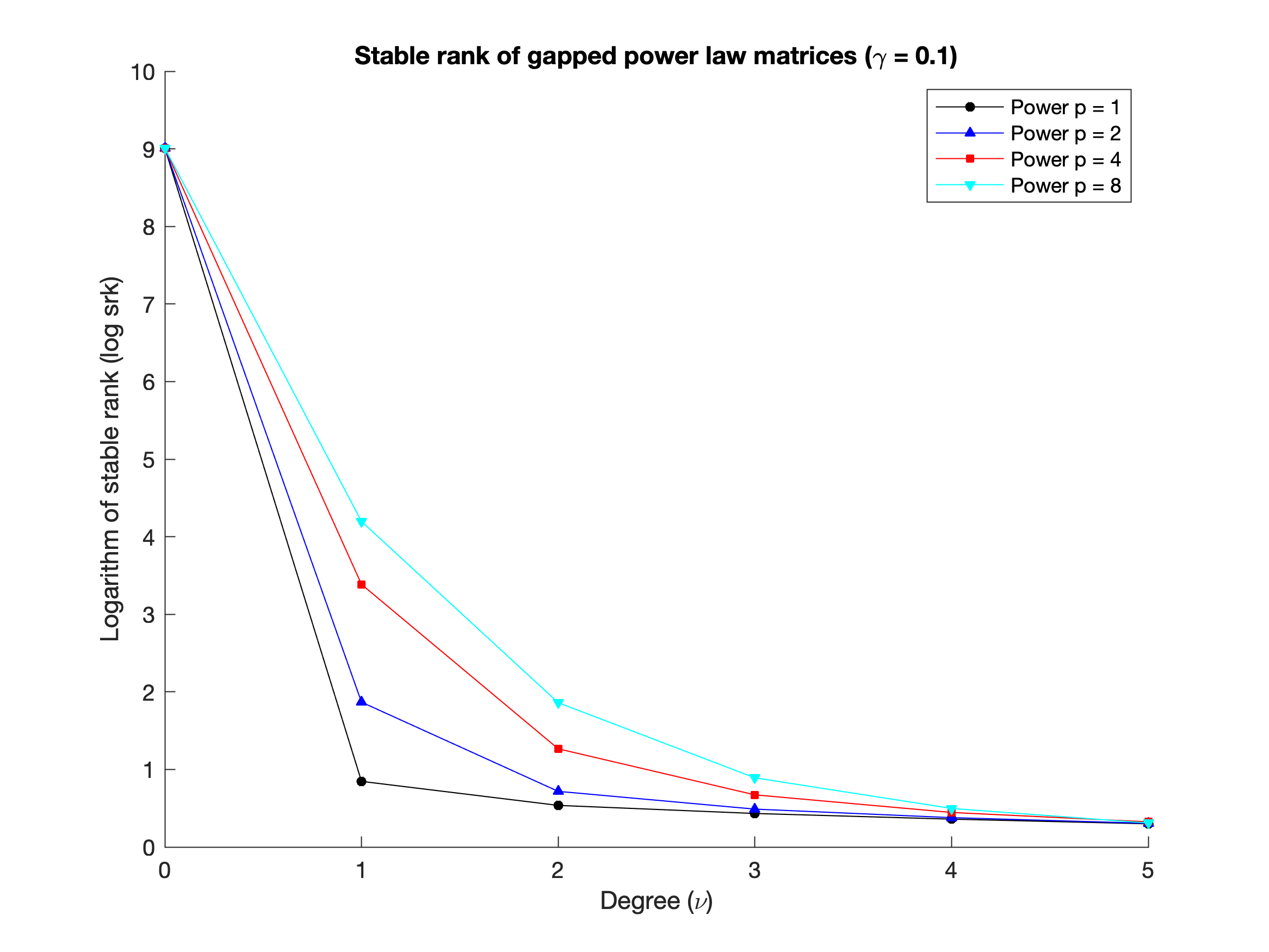}
\caption{Gapped power law matrices ($n = 8192$)} \label{fig:pow-srk}
\end{subfigure} \hspace{3pc}
\caption{\textbf{Logarithm of stable rank.}  These graphs display the natural logarithm
of the stable rank function, $\srank(\nu)$, for two types of matrices.  For gapped GOE matrices,
the data series illustrate the effect of increasing the dimension ($n$).
For gapped power law matrices, the series illustrate the effect
of decreasing the rate of tail decay ($p$).
} \label{fig:srank}
\end{figure}

\subsection{Sample Paths: The Role of Block Size}
\label{sec:numerics-paths}

First, we explore the empirical probability
that the randomized block Krylov method can
achieve a sufficiently small error.

The first set of experiments focuses on the behavior
of the block Krylov method for a matrix
with a substantial spectral gap.  
For a fixed $1000 \times 1000$ gapped GOE matrix with $\gamma = 0.1$,
Figure~\ref{fig:numerics-gap} illustrates
sample paths of the relative error in estimating
the maximum eigenvalue as a function
of the total depth $q$ of the Krylov space.
We compare the performance as the block size $\ell$
varies.  Here is a summary of our observations:

\begin{figure}
\begin{center}
\begin{subfigure}{0.5\textwidth}
\includegraphics[width=\columnwidth]{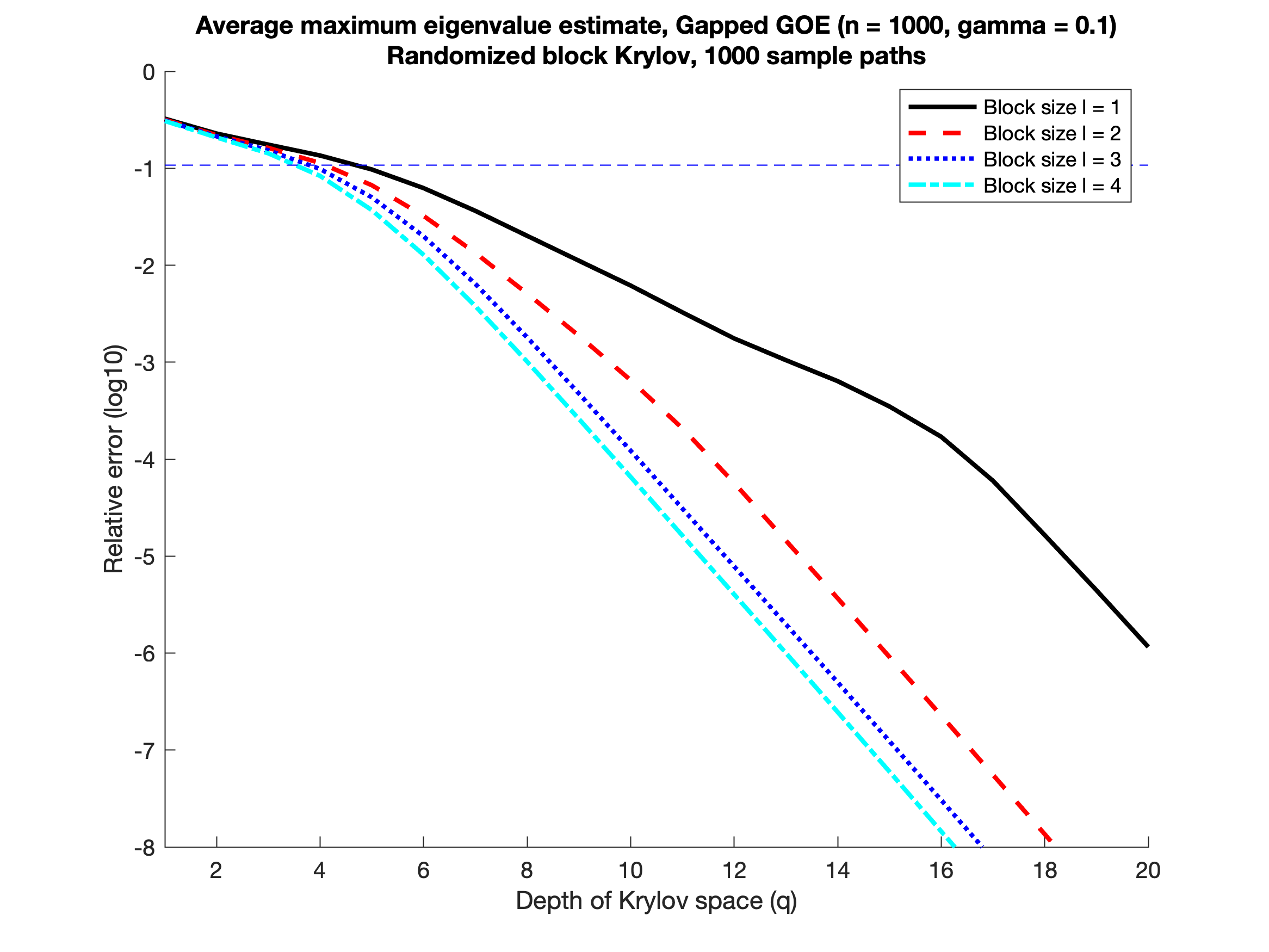}
\caption{Average error across block size}
\end{subfigure} \\ \vspace{1pc}
\begin{subfigure}{0.45\textwidth}
\includegraphics[width=\columnwidth]{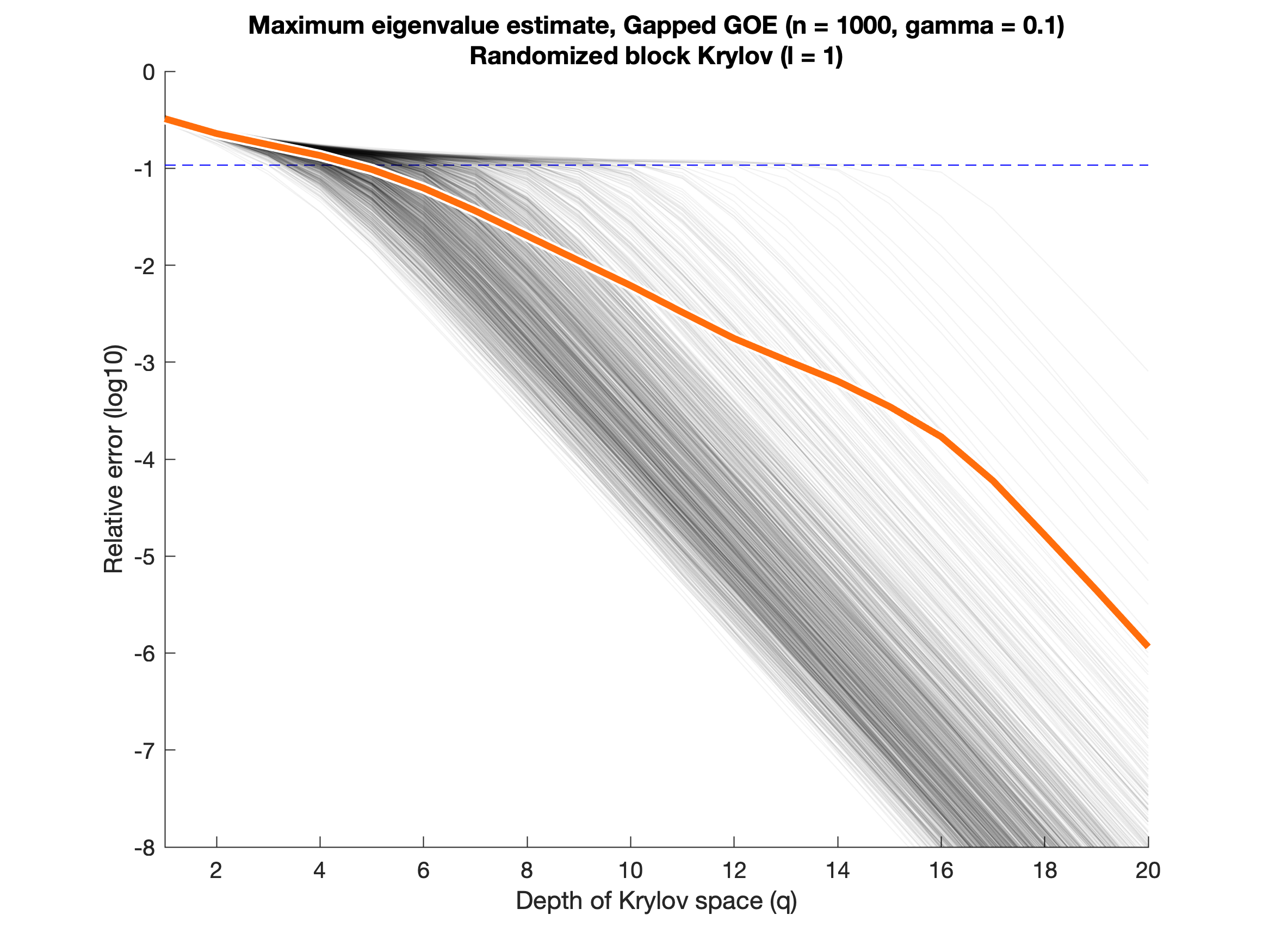}
\caption{Block size $\ell = 1$}
\end{subfigure}
\begin{subfigure}{0.45\textwidth}
\includegraphics[width=\columnwidth]{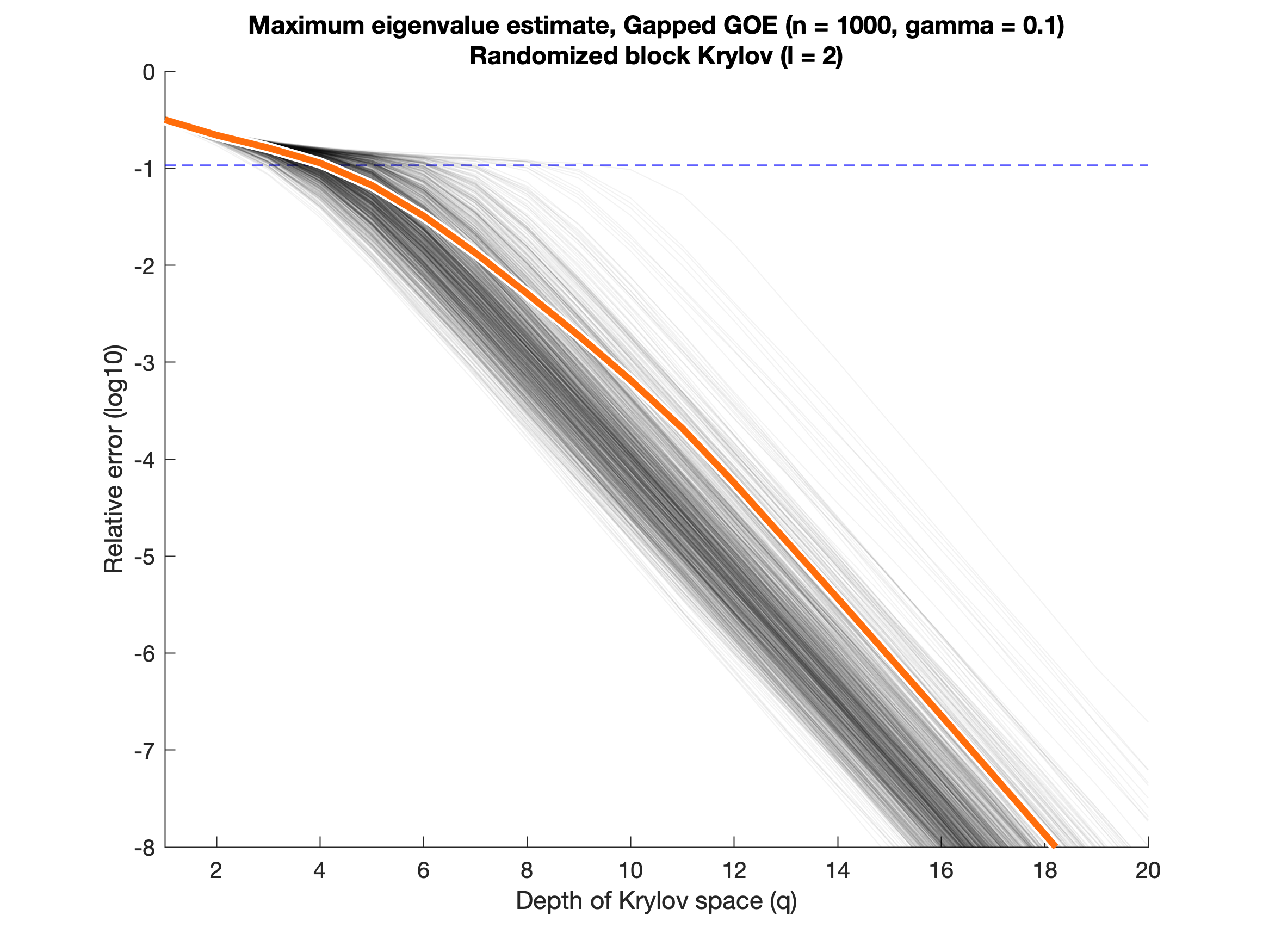}
\caption{Block size $\ell = 2$}
\end{subfigure} \\ \vspace{1pc}
\begin{subfigure}{0.45\textwidth}
\includegraphics[width=\columnwidth]{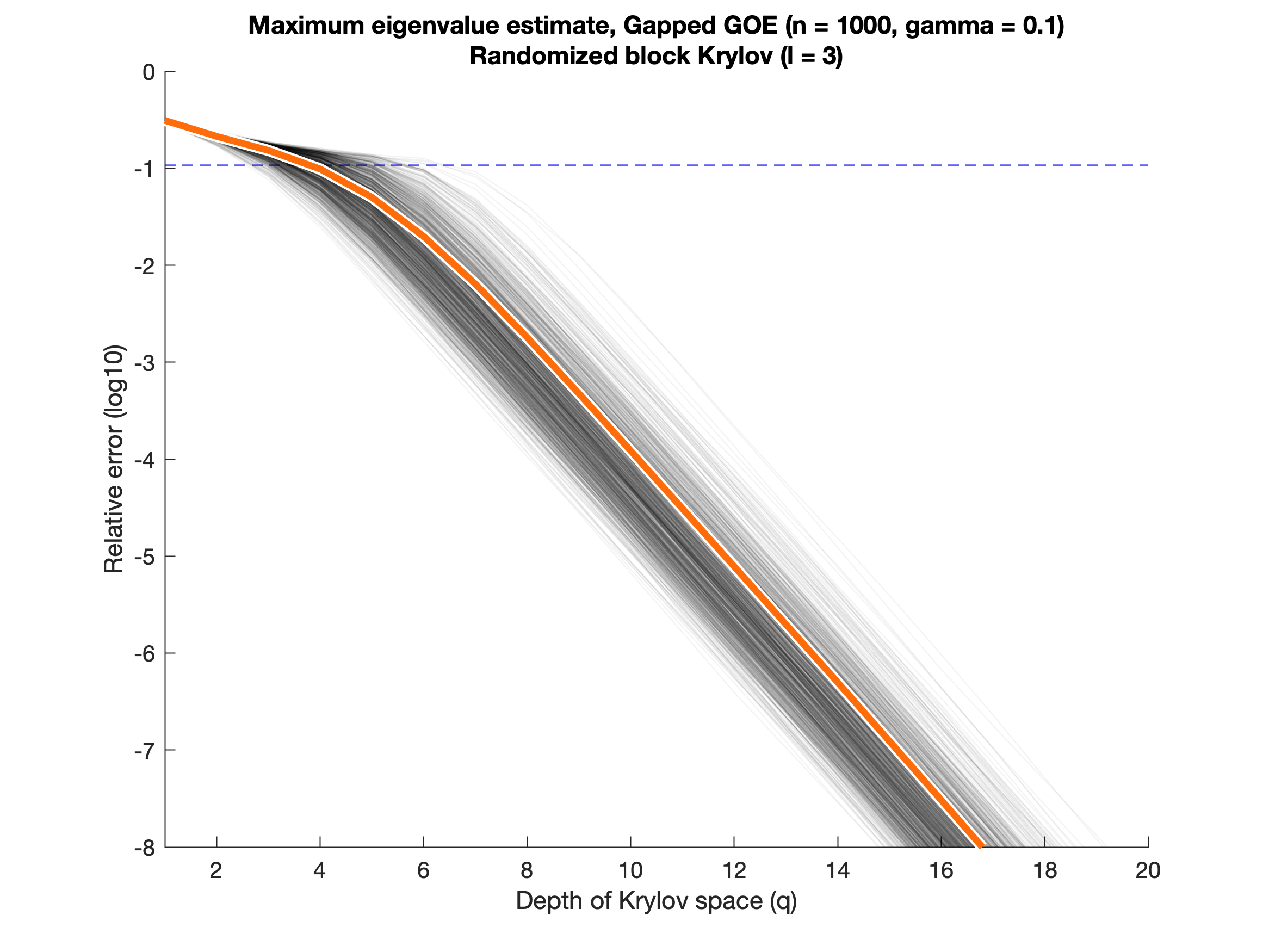}
\caption{Block size $\ell = 3$}
\end{subfigure}
\begin{subfigure}{0.45\textwidth}
\includegraphics[width=\columnwidth]{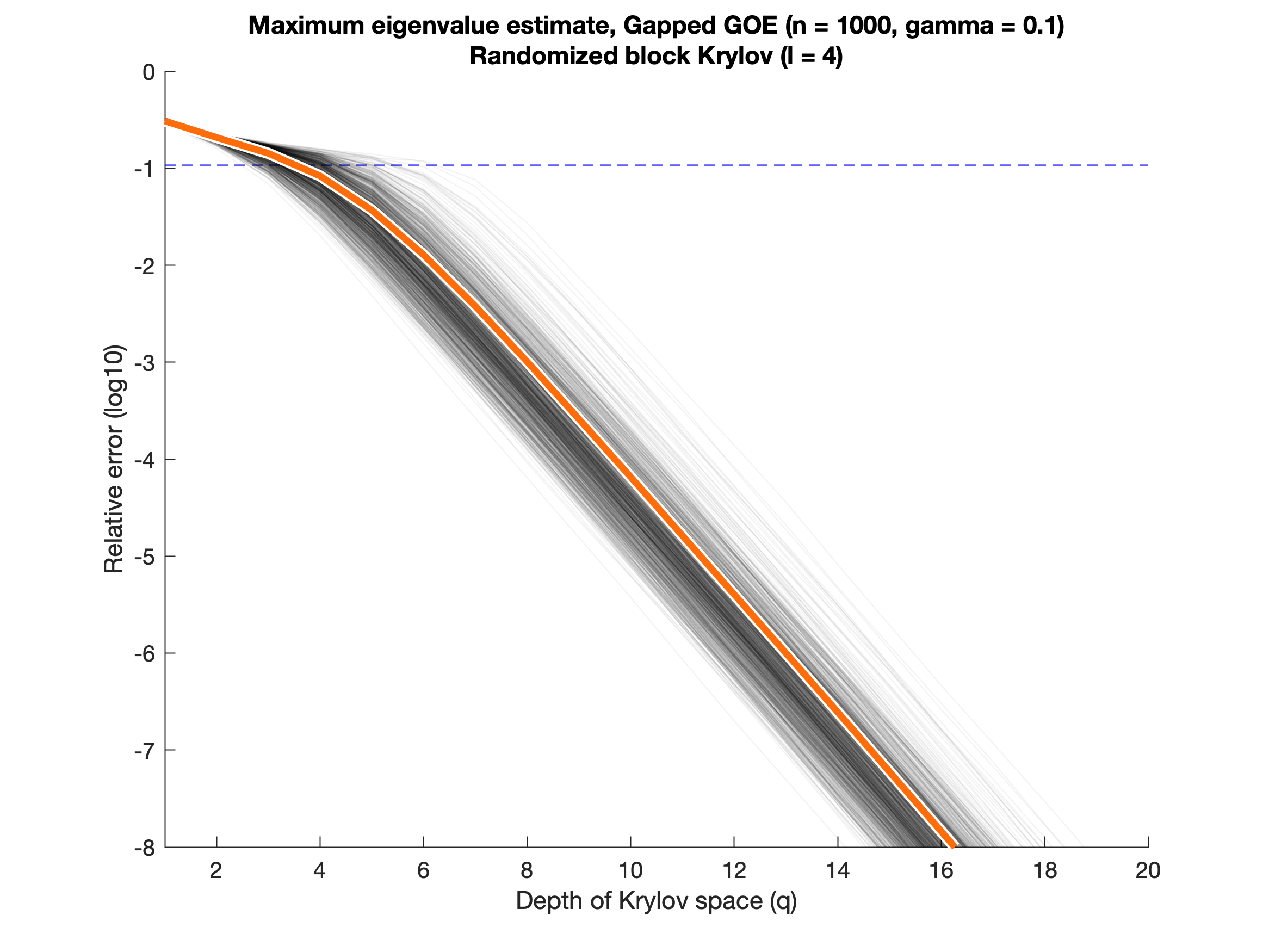}
\caption{Block size $\ell = 4$}
\end{subfigure}
\caption{\textbf{Sample paths, with spectral gap.}
The input matrix is a fixed $1000 \times 1000$ GOE matrix with an artificial
spectral gap $\gamma = 0.1$ (dashed blue line).  The bottom four panels show
the trajectory of the relative error (translucent hairlines)
in estimating the maximum eigenvalue via randomized block Krylov
for 1000 random test matrices with block size $\ell$ as a function of the depth $q$.
The average error (\textbf{not} average log-error) is marked with a heavy orange line.
The top panel compares the average error across block size.
See Section~\ref{sec:numerics-paths}.} \label{fig:numerics-gap}
\end{center}
\end{figure}

\begin{itemize}
\item	After a burn in period of about five steps, the error begins to decay exponentially,
as described in Theorem~\ref{thm:gap}.  For $\ell = 4$, we can graphically estimate that
the decay rate is about $\econst^{- 1.38 \, q }$, while the theory predicts a decay rate
at least $\econst^{-1.26 \, q}$.

\item	For block size $\ell = 1$, the average error decays at roughly \emph{half} the
rate achieved for block size $\ell \geq 2$.  For $\ell = 2$, the average error has a distinctive
profile; for $\ell = 3, 4$, the error is qualitatively similar.
These phenomena match the bounds in Theorem~\ref{thm:gap}(2).

\item	The sample paths give a clear picture of how the errors typically evolve.
Independent of the block size, most of the sample paths decay at the same rate.
As the block size increases, there is a slight reduction in the error (seen as
a shift to the left), but this improvement is both small and diminishing.

\item	The impact of the block size becomes evident when we look at the spread of
the sample paths.  As the block size increases, the error varies much less,
and this effect is sharpened by increasing the block size.  The
apparent reason is that the Krylov method can \textit{misconverge}: it %
locks onto the \textit{second largest} eigenvalue (dashed blue line),
and it may not %
find a component of the maximum eigenvector for a significant number of iterations.
For block size $\ell = 1$, this tendency is strong enough to destroy
the rate of average convergence.  For larger block sizes, the likelihood and duration
of misconvergence both decrease.
\end{itemize}

\begin{figure}
\begin{center}
\begin{subfigure}{0.5\textwidth}
\includegraphics[width=\columnwidth]{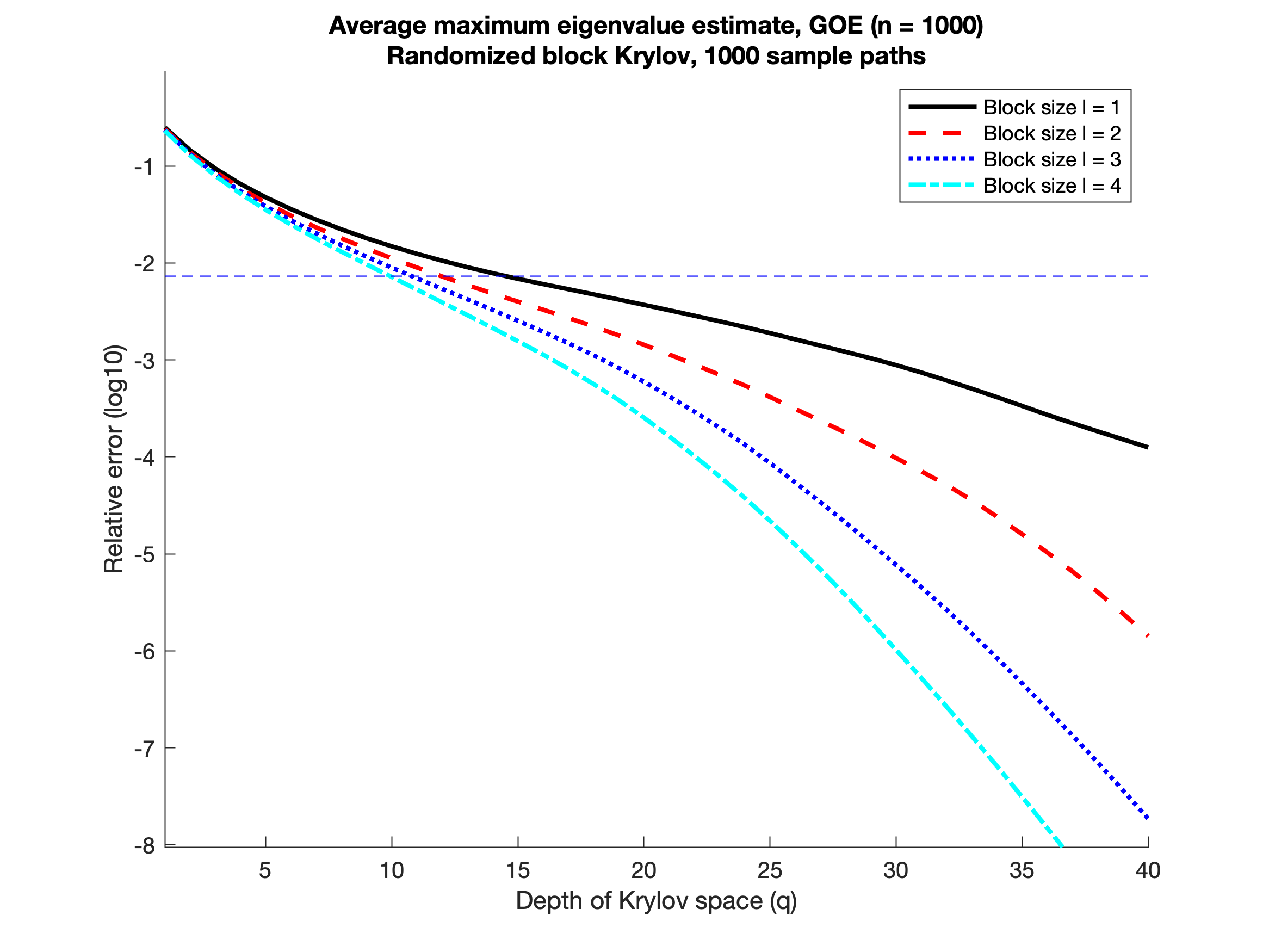}
\caption{Average error across block size}
\end{subfigure}  \\ \vspace{1pc}
\begin{subfigure}{0.45\textwidth}
\includegraphics[width=\columnwidth]{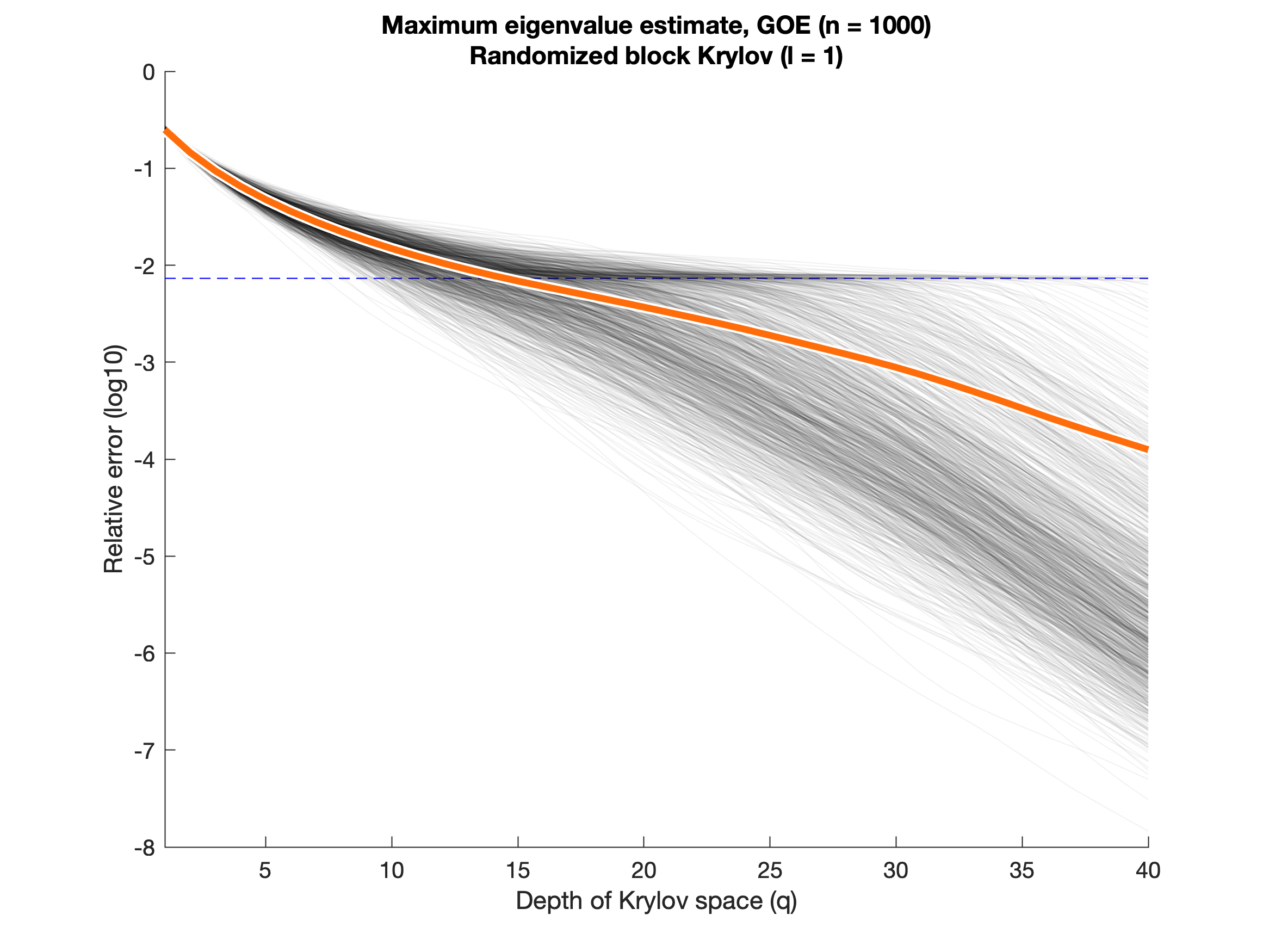}
\caption{Block size $\ell = 1$}
\end{subfigure}
\begin{subfigure}{0.45\textwidth}
\includegraphics[width=\columnwidth]{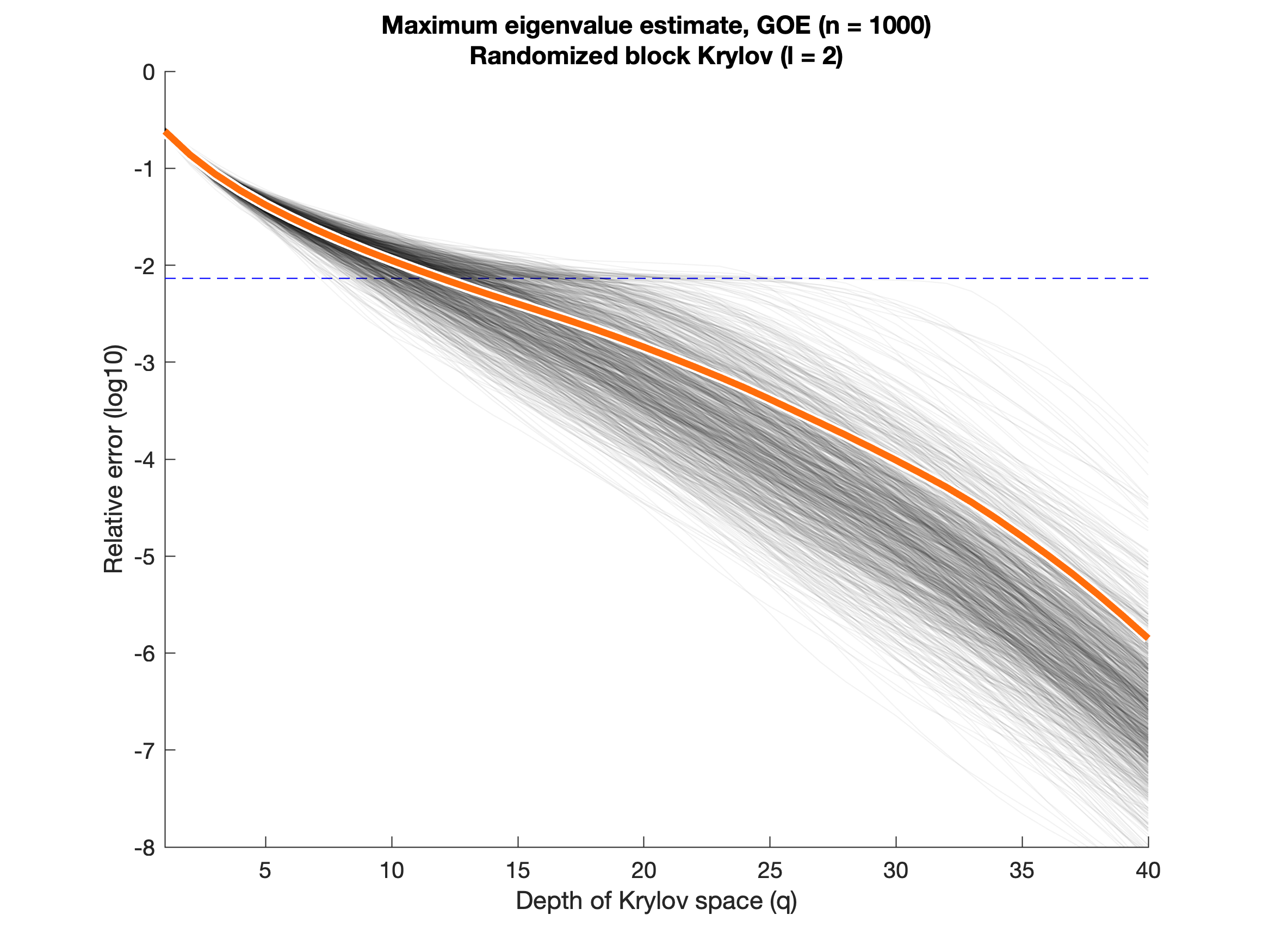}
\caption{Block size $\ell = 2$}
\end{subfigure} \\ \vspace{1pc}
\begin{subfigure}{0.45\textwidth}
\includegraphics[width=\columnwidth]{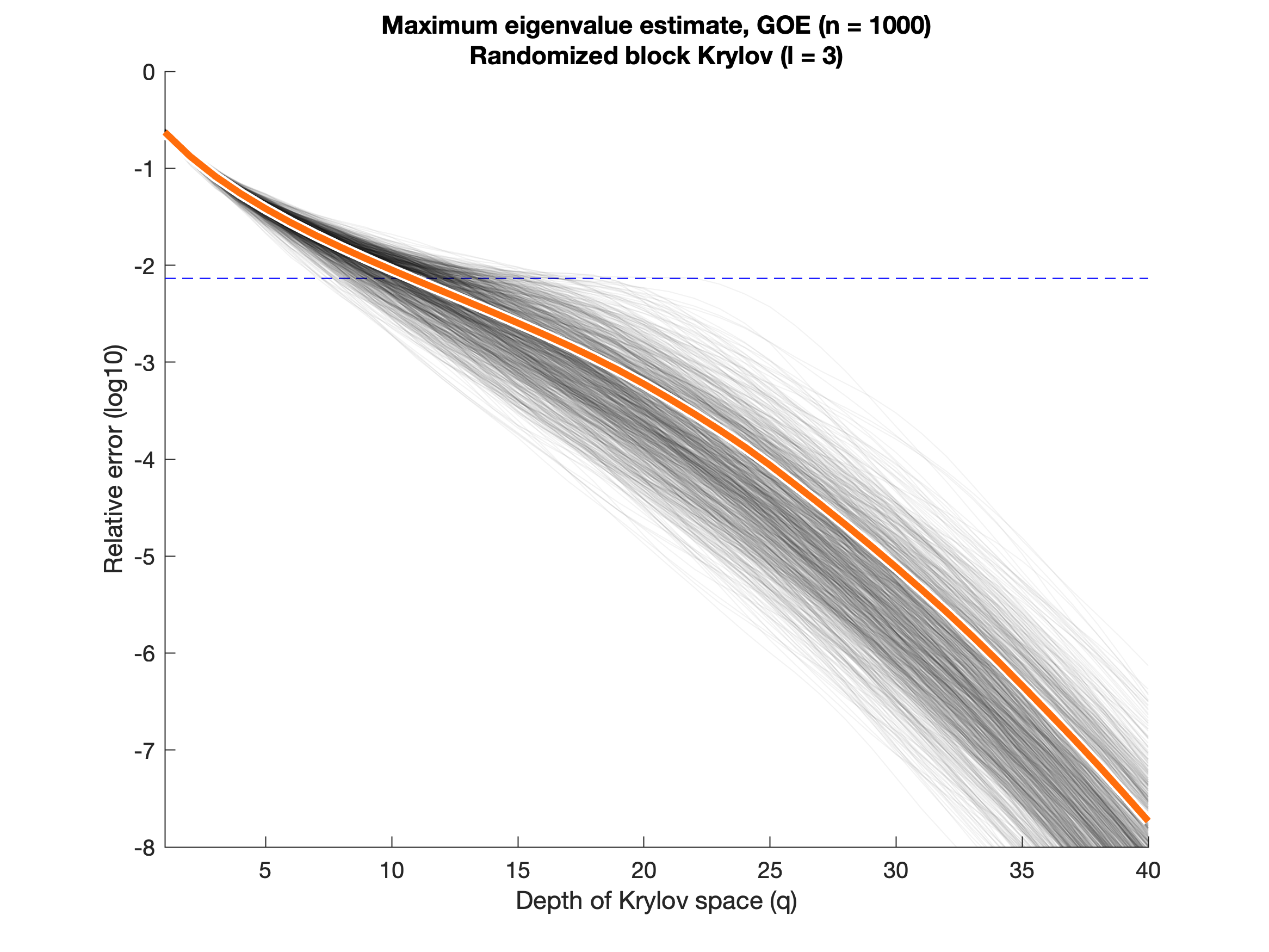}
\caption{Block size $\ell = 3$}
\end{subfigure}
\begin{subfigure}{0.45\textwidth}
\includegraphics[width=\columnwidth]{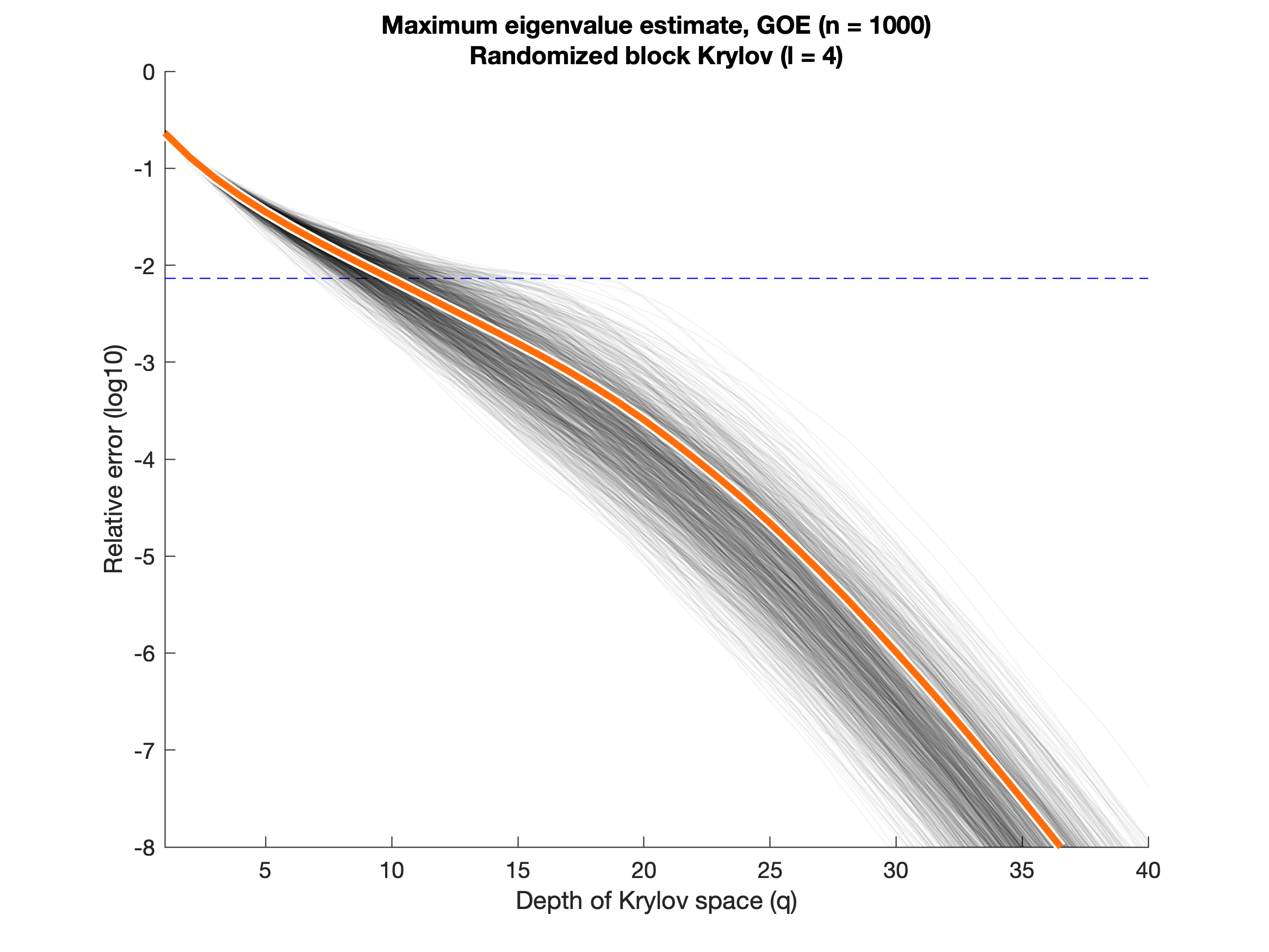}
\caption{Block size $\ell = 4$}
\end{subfigure}
\caption{\textbf{Sample paths, no spectral gap.}
The input matrix is a fixed $1000 \times 1000$ GOE matrix with spectral
gap $\gamma \approx 0.0073$ (dashed blue line).  The bottom four panels illustrate
the trajectory of the relative error (translucent hairlines)
in estimating the maximum eigenvalue via the randomized block Krylov method
for 1000 random test matrices with block size $\ell$ as a function of the depth $q$.
The average error (\textbf{not} average log-error) is marked with a heavy orange line.
The top panel compares the average error across block size.
See Section~\ref{sec:numerics-paths}.}  \label{fig:numerics-gapfree}
\end{center}
\end{figure}

The second set of experiments addresses the performance
of the block Krylov method for a matrix
with a small spectral gap.  
For a fixed $1000 \times 1000$ GOE matrix (with $\gamma \approx 0.0073$),
Figure~\ref{fig:numerics-gapfree} shows
sample paths of the relative error in estimating
the maximum eigenvalue as a function
of the total depth $q$ of the Krylov space
and the block size $\ell$.
Let us highlight a few observations.

\begin{itemize}
\item	Initially, the average error decays roughly as $q^{-2}$, as predicted by Theorem~\ref{thm:gapfree}.
Once the error is sufficiently small (around $10^{-3}$), the decay becomes (super)exponential as indicated
by Theorem~\ref{thm:gap}.

\item	In the polynomial decay regime $(q \lessapprox 20)$, the block size affects the average
error weakly, as suggested by Theorem~\ref{thm:gapfree}.  In the exponential decay regime,
the block size plays a much more visible role, but the theory does not fully capture this effect.

\item	Regardless of the block size, the bulk of the sample paths decay more quickly than
the average error.  The block size has less of an effect on the typical error than on the
average error.  Nevertheless, for larger block size, we quickly achieve more digits of accuracy.

\item	As in the first experiment, the block size has a major impact on the variability of the error.
As we increase the block size, the sample paths start to cluster sharply, so the typical
error and the average error align with each other.  Misconvergence is also
visible here, and it can be mitigated by increasing the block size.
\end{itemize}

\subsection{Burn-in: The Role of Tail Content}

Next, we examine how the tail content affects the burn-in period
for the randomized block Krylov method.

Let us consider how the burn-in increases with the dimension
of a matrix with limited spectral decay.
For $n \times n$ gapped GOE matrices with $\gamma = 0.1$,
Figure~\ref{fig:burn-dimen} shows how the average error
evolves as a function of the depth $q$ of the Krylov
space with block size $\ell = 2$.  The main point is
that the error initially stagnates before starting to
decay exponentially.  The stagnation period increases
in proportion to the logarithm of the dimension $n$.
As we saw in Figure~\ref{fig:goe-srk},
this same behavior is visible in the size of
$\log \srank(\nu)$ for small values of $\nu$.

Now, we look at how the burn-in depends on the rate of decay
of the eigenvalues of a matrix.  For an $8192 \times 8192$
gapped power law matrix with $\gamma = 0.1$,
Figure~\ref{fig:burn-decay} charts how the
average error decays as a function of the depth $q$
for block size $\ell = 2$.  We observe that the error
initially decays exponentially at a slow rate.
After a burn-in period, the error begins to decline
at a faster exponential rate.  The length of
the initial trajectory increases in proportion with
the logarithm of the order $p$ of the power law.
From Figure~\ref{fig:pow-srk}, we detect that
this speciation reflects the spread of $\log \srank(\nu)$
for small values of $\nu$.

\begin{figure}
\begin{subfigure}{0.49\textwidth}
\includegraphics[width=\columnwidth]{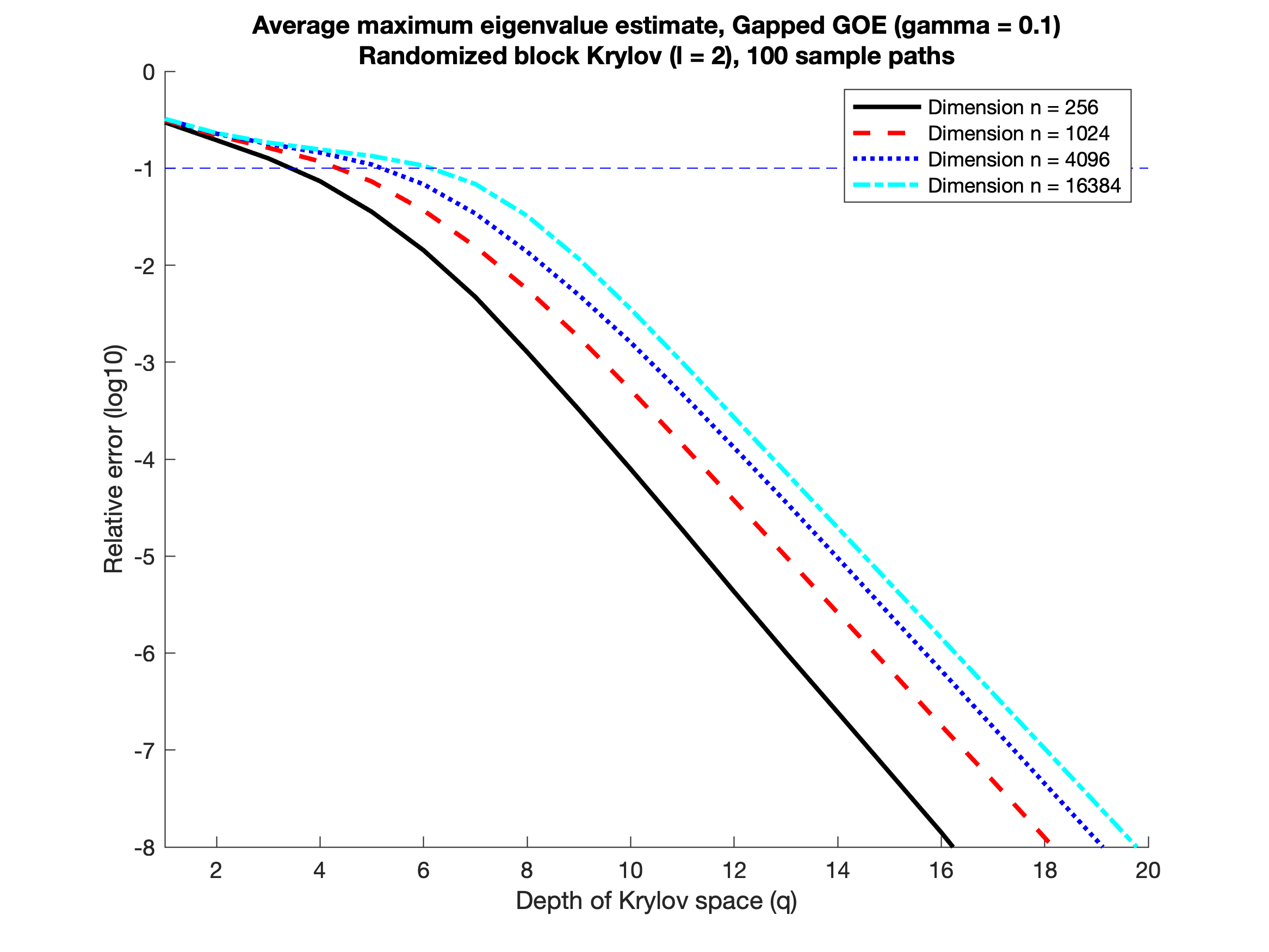}
\caption{Role of dimension} \label{fig:burn-dimen}
\end{subfigure}
\begin{subfigure}{0.49\textwidth}
\includegraphics[width=\columnwidth]{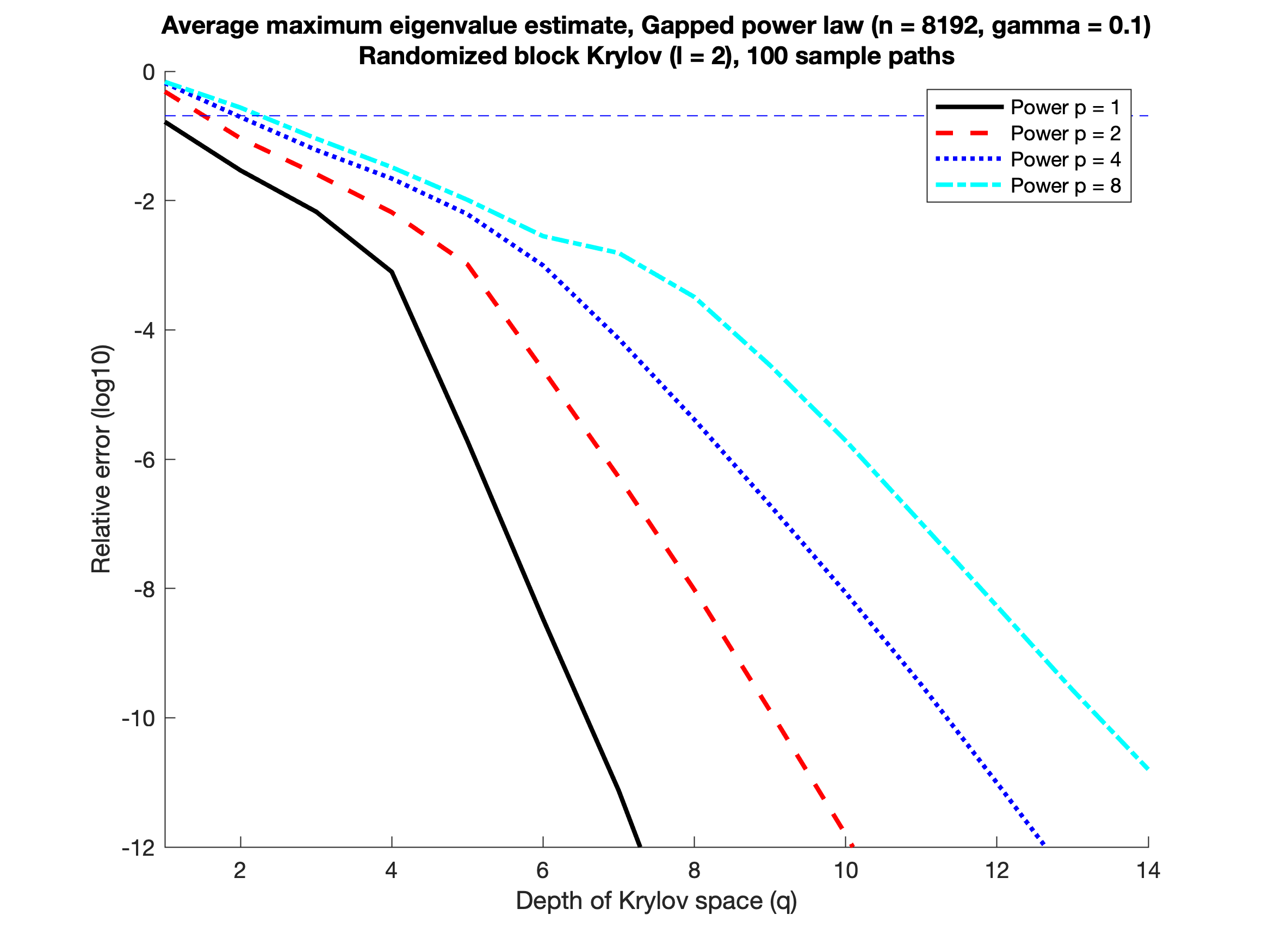}
\caption{Role of tail decay} \label{fig:burn-decay}
\end{subfigure}
\caption{\textbf{Burn-in as a function of tail content.} These plots show
the average error in estimating
the maximum eigenvalue as a function of depth $q$ with block size $\ell = 2$.
In the left panel, we consider gapped GOE matrices of increasing dimension $n$.
In the right panel, we consider gapped power law matrices with decreasing rate $p$
of tail decay.}
\end{figure}

\subsection{Conclusions}

The numerical experiments presented in this section
confirm many of our theoretical predictions about
the behavior of randomized block Krylov methods
for estimating the maximum eigenvalue of a symmetric
matrix.
Increasing the block size results in a dramatic
reduction in the probability of committing a
large error in the estimate.
The underlying mechanism is that simple Krylov method ($\ell = 1$)
is much more likely to misconverge
than a block method (even with $\ell = 2, 3$). 
The effect is so significant that block Krylov methods
can converge, on average, at twice the
rate of a simple Krylov method.  These
benefits counterbalance the increased
computational cost of a block method.

These facts have implications for design and
implementation of practical algorithms.
If the arithmetic cost is the driving concern,
then the simple randomized Krylov method
typically uses matrix--vector multiplications
more efficiently than the block methods.
But we have also seen that block methods can achieve a
specific target for the error (averaged over test matrices)
using a similar number of matrix--vector multiplies as the simple method.
In fact, in modern computing environments,
the actual cost (time, communication, energy)
of multiplying a matrix by several vectors may be
equivalent to a single matrix--vector multiplication,
in which case the block methods are the obvious choice.
Finally, when it is critical to limit the probability of failure,
then the block methods are clearly superior.

\section{History, Related Work, and Contributions}

Krylov subspace methods are a wide class of algorithms that use matrix--vector products
(``Krylov information'') to compute eigenvalues and eigenvectors and to solve linear systems.
These methods are especially appropriate in situations where we can only interact with a matrix
through its action on vectors.  In this treatment, we only discuss Krylov methods
for spectral computations.  Some of the basic algorithms in this area are the
power method, the inverse power method, subspace iteration, the Lanczos method,
the block Lanczos method, and the band Lanczos method.  See the
books~\cite{Par98:Symmetric-Eigenvalue,BDD+00:Templates-Solution,Saa11:Numerical-Methods,GVL13:Matrix-Computations}
for more background and details.

\subsection{Simple Krylov Methods}

Simple Krylov methods are algorithms based on a Krylov subspace $K_q(\mtx{A}; \vct{b})$
constructed from a single starting vector $\vct{b}$.  That is, the block size $\ell = 1$.

The power method, which dates to the 19th century, is probably the earliest algorithm
that relies on Krylov information to compute eigenvalues and eigenvectors of a symmetric matrix.
The power method is degenerate in the sense that it only keeps the highest-order term
in the Krylov subspace.

In the late 1940s, Lanczos~\cite{Lan50:Iteration-Method} developed a sophisticated
Krylov subspace method for solving the symmetric eigenvalue problem.
(More precisely, the Lanczos method uses a three-term recurrence to compute a basis
for the Krylov subspace so that the compression of the input matrix to the Krylov subspace
is tridiagonal.)
In exact arithmetic, the Lanczos estimate of the maximum eigenvalue of a symmetric
matrix coincides with $\xi_{\max}(\mtx{A}; \vct{b}; q)$ for a fixed vector $\vct{b}$.
On the other hand, the Lanczos method has complicated behavior in finite-precision arithmetic;
for example, see Meurant's book~\cite{Meu06:Lanczos-Conjugate}.

The first analysis of the Lanczos method with a deterministic starting vector $\vct{b}$ dates to
the work of Lanczos~\cite{Lan50:Iteration-Method}.
Kaniel, Paige, and Saad also made major theoretical contributions in the 1970s and 1980s;
see~\cite{Par98:Symmetric-Eigenvalue,Saa11:Numerical-Methods} for details and references.
In the 1980s, Nemirovsky, Yudin, and Chou showed that Krylov subspace methods are the optimal
deterministic algorithms for solving the symmetric eigenvalue problem, assuming we only
have access to the matrix via matrix--vector multiplication;
see~\cite{NY83:Problem-Complexity,Cho87:Optimality-Krylov,Nem91:Optimality-Krylovs}.

The burn-in period for Krylov methods has been observed in many previous works,
including~\cite{VV86:Rate-Convergence,LS05:GMRES-Convergence,BES05:Convergence-Polynomial}.
The length of the burn-in period depends on the proportion of the energy
in the test vector that is captured by the true invariant subspace.

The main contribution of the paper~\cite{VV86:Rate-Convergence}
is the observation that Krylov methods may exhibit superlinear convergence.
The explanation for this phenomenon is that the optimal polynomial
can create an artificial spectral gap by annihilating undesirable eigenvalues.
This kind of behavior is visible in Figure~\ref{fig:numerics-gapfree}.

\subsection{Random Starting Vectors}

Practitioners have often suggested using randomized variants of Krylov subspace methods.
That is, the starting vector $\vct{b}$ is chosen at random.
Historically, randomness was just used to avoid the situation where the
starting vector $\vct{b}$ is orthogonal to the leading invariant subspace of the matrix.

Later, deeper justifications for random starting vectors appeared.
The first substantive theoretical analysis of a randomized Krylov method appears
in Dixon's paper~\cite{Dix83:Estimating-Extremal} on the power method with a random starting vector.
We believe that this is the first paper to recognize that Krylov methods
can be successful without the presence of a spectral gap.

In 1992, Kuczy{\'n}ski \& Wo{\'z}niakowski published an analysis~\cite{KW92:Estimating-Largest}
of the Lanczos method with a random starting vector.
Their work highlighted the benefits of randomization, and it provided a clear explanation
of the advantages of using full Krylov information instead of the power method.  
See the papers~\cite{KW94:Probabilistic-Bounds,DM97:Randomized-Error,LW98:Estimating-Largest}
for further work in this direction.

The recent papers~\cite{SER17:Gap-Strict-Saddles,SER18:Tight-Query}
contain lower bounds on the performance of randomized algorithms
for the symmetric eigenvalue problem that use Krylov information.

\subsection{Block Krylov Methods}

Block Krylov subspace methods use multiple starting vectors to generate the Krylov subspace, instead of just one.
In other words, the algorithms form a Krylov subspace $K_q(\mtx{A}; \mtx{B})$, where $\mtx{B}$ is a matrix.
These methods were developed in the late 1960s and 1970s in an effort to resolve
multiple eigenvalues more reliably.
The block analog of the power method is
called subspace iteration; see the books~\cite{Par98:Symmetric-Eigenvalue,Saa11:Numerical-Methods} for discussion.

There are also block versions of the Lanczos method, which were developed
by Cullum \& Donath~\cite{CD74:Block-Generalization} and Golub \& Underwood~\cite{GU77:Block-Lanczos}.
(More precisely, the block Lanczos method uses a recurrence to compute a basis
for the block Krylov subspace so that the compression of the input matrix
to the block Krylov subspace is block tridiagonal.)
In exact arithmetic, the block Lanczos estimate of the maximum eigenvalue of a symmetric matrix coincides with
$\xi_{\max}(\mtx{A}; \mtx{B}; q)$ for a fixed matrix $\mtx{B}$.

Most of the early work on block Krylov subspace methods focuses on the case
where the block size $\ell$ is small, while the depth $q$ of the Krylov space
is moderately large.  This can lead to significant problems with numerical stability,
especially in the case where we use a recurrence to perform orthogonalization.
As a consequence, full orthogonalization is usually recommended.
Furthermore, most of the existing analysis of block Krylov methods is deterministic;
for example, see~\cite{Saa80:Rates-Convergence,LZ15:Convergence-Block}.

\subsection{Randomized Block Krylov Methods}

Over the last decade, randomized block Krylov subspace methods have emerged
as a powerful tool for spectral computations on large matrices.
These algorithms use a Krylov subspace $K_q(\mtx{A}; \mtx{B})$ generated by
a \emph{random} test matrix $\mtx{B}$.  

In contrast with earlier block Krylov algorithms,
contemporary methods use a much larger block size $\ell$ and a much smaller depth $q$.
Furthermore, the randomness plays a central role in supporting performance guarantees
for the algorithms.

Most of the recent literature concerns the problem of computing a low-rank approximation
to a large matrix, rather than estimating eigenvalues or invariant subspaces.
Some of the initial work on randomized algorithms
for matrix approximation appears in~\cite{PRTV98:Latent-Semantic,
FKV98:Fast-Monte-Carlo,DKM06:Fast-Monte-Carlo-II,MRT11:Randomized-Algorithm}.
Randomized subspace iteration was proposed in~\cite{RST09:Randomized-Algorithm}
and further developed in~\cite{HMT11:Finding-Structure}.
Randomized block Krylov methods that use the full block Krylov subspace were
proposed in the papers~\cite{RST09:Randomized-Algorithm,HMST11:Algorithm-Principal};
see also~\cite{DIKM18:Structural-Convergence}.
See~\cite{HMT11:Finding-Structure} for more background and history.

There is some theoretical and empirical evidence~\cite{HMST11:Algorithm-Principal,MM15:Randomized-Block}
that randomized block Krylov methods can produce low-rank matrix approximations
with higher accuracy and with less computation than randomized subspace iteration.

\subsection{Analysis of Randomized Block Krylov Methods}

There is a growing body of literature that develops theoretical performance
guarantees for randomized block Krylov methods.
The papers~\cite{RST09:Randomized-Algorithm,HMT11:Finding-Structure,
Gu15:Subspace-Iteration,MM15:Randomized-Block}
contain theoretical analyses of randomized subspace iteration.
The papers~\cite{MM15:Randomized-Block,WZZ15:Improved-Analyses,DIKM18:Structural-Convergence}
contain theoretical analysis of randomized methods that use the
full block Krylov space.  These works all focus on low-rank matrix approximation.

\subsection{Contemporary Work}

After the initial draft of this paper and its companion~\cite{Tro18:Analysis-Randomized-TR}
was completed, several other research groups
released new work on randomized (block) Krylov methods. %
The paper~\cite{DI19:Low-Rank-Matrix} shows that randomized block Lanczos can produce
excellent low-rank matrix approximations, even without the presence of a spectral gap.
The paper~\cite{YGL18:Superlinear-Convergence} demonstrates that randomized block Lanczos gives
excellent estimates for the singular values of a general matrix with spectral decay.

\subsection{Contributions}

We set out to develop highly refined bounds for
the behavior of randomized block Krylov methods that use the full Krylov subspace.
Our aim is to present useful and informative results in the spirit of Saad~\cite{Saa80:Rates-Convergence},
Kuczi{\'n}ski \& Wo{\'z}niakowski~\cite{KW92:Estimating-Largest},
and Halko et al.~\cite{HMT11:Finding-Structure}.
Our research has a number of specific benefits over prior work.

\begin{itemize}
\item	We have shown that randomized block Krylov methods have exceptional
performance for matrices with spectral decay.  In fact, for matrices with
polynomial spectral decay, we can often obtain accurate estimates
even when the block Krylov subspace has \emph{constant} depth.

\item	We have obtained detailed information about the role of the block size $\ell$.
In particular, increasing the block size drives down failure probabilities exponentially fast.

\item	The companion paper~\cite{Tro18:Analysis-Randomized-TR} gives the first results on the performance of randomized
block Krylov methods for the symmetric eigenvalue problem.

\item	Our bounds have explicit and modest constants, which gives the bounds
some predictive power.
\end{itemize}

\noindent
We hope that these results help clarify the benefits of randomized block Krylov methods.
We also hope that our work encourages researchers to develop new implementations
of these algorithms that fully exploit contemporary computer architectures.

\section{Preliminaries}

Before we begin the proofs of the main results, Theorems~\ref{thm:gapfree} and~\ref{thm:gap},
we present some background information.
In Section~\ref{app:rot-invar}, we justify the claim that
the test matrix should have a uniformly random range.
In Section~\ref{app:chebyshev}, we introduce the Chebyshev
polynomials of the first and second type, and we develop
the properties that we need to support our arguments.
Expert readers may wish to skip to Section~\ref{sec:error-formula}
for proofs of the main results.

\subsection{Rotationally Invariant Distributions}
\label{app:rot-invar}

We wish to estimate spectral properties of a matrix
using randomized block Krylov information.
In particular, we aim to establish probabilistic upper bounds
on the error in these spectral estimates for any input matrix.
This section contains a general argument that clarifies why we
ought to use a random test matrix with a rotationally invariant range
in these applications.

\begin{proposition}[Uniformly Random Range] \label{prop:random-range}
Consider any bivariate function
$f : \R^{n \times n} \times \R^{n \times \ell} \to \R$
that is orthogonally invariant:
$$
f( \mtx{A}; \mtx{B}) = f( \mtx{U A U}^*; \mtx{U B} )
	\quad\text{for each orthogonal $\mtx{U} \in \R^{n \times n}$.}
$$
Fix a symmetric $n \times n$ matrix $\mtx{\Lambda}$, and consider the orthogonal orbit
$$
\mathcal{A} := \mathcal{A}(\mtx{\Lambda}) := \big\{ \mtx{U\Lambda U}^* : \text{$\mtx{U} \in \R^{n \times n}$ orthogonal} \big\}.
$$
Let $\mtx{B} \in \R^{n \times \ell}$ be a random matrix.
Let $\mtx{V} \in \R^{n \times n}$ be a uniformly random orthogonal matrix,
drawn independently from $\mtx{B}$.  Then
$$
\begin{aligned}
\max_{\mtx{A} \in \mathcal{A}} {} \Expect_{\mtx{V}, \mtx{B}} f(\mtx{A}; \mtx{VB})
	&\leq \max_{\mtx{A} \in \mathcal{A}} {} \Expect_{\mtx{B}} f(\mtx{A}; \mtx{B}). %
\end{aligned}
$$
\end{proposition}

\begin{proof}
By orthogonal invariance of $f$,
$$
\begin{aligned}
\max_{\mtx{A} \in \mathcal{A}} {} \Expect_{\mtx{V}, \mtx{B}} f(\mtx{A}; \mtx{VB})
	&= \max_{\mtx{A} \in \mathcal{A}} {} \Expect_{\mtx{V}, \mtx{B}} f(\mtx{V}^* \mtx{A} \mtx{V}; \mtx{B}) \\
	&\leq \Expect_{\mtx{V}} \max_{\mtx{A} \in \mathcal{A}} {} \Expect_{\mtx{B}} f(\mtx{V}^* \mtx{A} \mtx{V}; \mtx{B})
	= \max_{\mtx{A} \in \mathcal{A}} {} \Expect_{\mtx{B}} f(\mtx{A}; \mtx{B}).
\end{aligned}
$$
The inequality is Jensen's, and the last identity follows from the definition
of the class $\mathcal{A}$.
\end{proof}

Now, consider the problem of estimating the maximum eigenvalue of the worst matrix
with eigenvalue spectrum $\mtx{\Lambda}$ using a block Krylov method.
Define the orthogonally invariant%
\footnote{The relative error~\eqref{eqn:error-def} is orthogonally invariant because the eigenvalues of $\mtx{A}$ are orthogonally invariant and~\eqref{eqn:rotation-invar} states that the eigenvalue estimate~\eqref{eqn:eigenvalue-estimate}
is also orthogonally invariant.}
function $f(\mtx{A}; \mtx{B}) =  \mathrm{err}(\xi_{\max}(\mtx{A}; \mtx{B}; q))$.
For any distribution on $\mtx{B}$ and for a uniformly random orthogonal matrix $\mtx{V}$,
Proposition~\ref{prop:random-range} states that
$$
\max_{\mtx{A} \in \mathcal{A}} {} \Expect \mathrm{err}(\xi_{\max}(\mtx{A}; \mtx{VB}; q))
	\leq \max_{\mtx{A} \in \mathcal{A}} {} \Expect \mathrm{err}(\xi_{\max}(\mtx{A}; \mtx{B}; q)).
$$
That is, the random test matrix $\mtx{VB}$ %
is better than the test matrix $\mtx{B}$ if we want to minimize
the worst-case expectation of the error;
the same kind of bound holds for tail probabilities.
We surmise that the test matrix $\mtx{B}$ should have a uniformly random range.
Moreover, because of the co-range invariance~\eqref{eqn:range-invar},
we can select \emph{any} distribution with uniformly random range,
such as the standard normal matrix $\mtx{\Omega}$.

\subsection{Chebyshev Polynomials}
\label{app:chebyshev}

In view of~\eqref{eqn:krylov-poly}, we can interpret Krylov subspace methods
as algorithms that compute best polynomial approximations.
The analysis of Krylov methods often involves
a careful choice of a specific polynomial that witnesses the behavior
of the algorithm.  Chebyshev polynomials arise naturally in this
connection because they are the solutions to minimax approximation problems.

This section contains the definitions of the Chebyshev polynomials of the
first and second type, and it derives some key properties of these polynomials.
We also construct the specific polynomials that arise in our analysis.
Some of this material is drawn from the paper~\cite{KW92:Estimating-Largest}
by Kuczy{\'n}ski \& Wo{\'z}niakowski.
For general information about Chebyshev polynomials, we refer the reader
to \cite{AS64:Handbook-Mathematical,NIST10:Handbook-Mathematical}.

\subsubsection{Chebyshev Polynomials of the First Kind}

We can define the Chebyshev polynomials of the first kind via the formula
\begin{equation} \label{eqn:Tp}
T_p(s) := \frac{1}{2} \left[ \left(s + \sqrt{s^2 - 1}\right)^p + \left(s - \sqrt{s^2 - 1}\right)^p \right]
	\quad\text{for $s \in \R$ and $p \in \Z_+$.}
\end{equation}
Using the binomial theorem, it is easy to check that this expression coincides
with a polynomial of degree $p$ with real coefficients.

We require two properties of the Chebyshev polynomial $T_p$.
First, it satisfies a uniform bound on the unit interval:
\begin{equation} \label{eqn:Tp-minmax}
\abs{ T_p(s) } \leq 1
\quad\text{for $\abs{s} \leq 1$.}
\end{equation}
Indeed, it is well-known that $2^{-p} T_p$ is the unique monic polynomial of degree $p$
with the least maximum value on the interval $[-1,+1]$.
The simpler result~\eqref{eqn:Tp-minmax} is an immediate consequence of the representation
$$
T_p(s) = \cos\big( p \cos^{-1} (s) \big)
\quad\text{for $\abs{s} \leq 1$.}
$$
The latter formula follows from~\eqref{eqn:Tp} after we apply de Moivre's theorem
for complex exponentiation.

Second, the Chebyshev polynomial grows quickly outside of the unit interval:
\begin{equation} \label{eqn:Tp-growth}
T_p\left( \frac{1+s}{1-s} \right)
	\geq \frac{1}{2} \left( \frac{1 + \sqrt{s}}{1 - \sqrt{s}} \right)^p
	\quad\text{for $0 \leq s < 1$.}
\end{equation}
This estimate is a direct consequence of the definition~\eqref{eqn:Tp}.

\subsubsection{The Attenuation Factor}

Let $\beta \in [0, 1]$ be a parameter, and define the quantity
\begin{equation} \label{eqn:delta}
\delta := \delta(\beta) = \frac{1 - \sqrt{1 - \beta}}{1 + \sqrt{1 - \beta}}.
\end{equation}
This definition is closely connected with the growth properties of $T_p$.
We can bound the attenuation factor in two ways:
\begin{equation} \label{eqn:delta-bd}
\delta \leq \econst^{- 2 \sqrt{1 - \beta}}
\quad\text{and}\quad
\delta \leq \beta \cdot 2^{-2\sqrt{1 - \beta}}.
\end{equation}
These numerical inequalities can be justified using basic calculus.
The first is very accurate for $\beta \approx 1$, while the second
is better across the full range $\beta \in [0, 1]$.

\subsubsection{First Polynomial Construction}
\label{app:first-poly}

Choose a nonnegative integer partition $q = q_1 + q_2$.  Consider the polynomial
\begin{equation} \label{eqn:phi-def}
\phi_{\beta, q_1,q_2}(s) := \frac{s^{q_1} T_{q_2}((2/\beta) s - 1)}{T_{q_2}((2/\beta) - 1)}
\quad\text{for $s \in \R$.}
\end{equation}
The polynomial $\phi_{\beta, q_1,q_2}$ has degree $q$, and it is normalized so that
it takes the value one at $s = 1$.
It holds that
\begin{equation} \label{eqn:phi-bound}
\phi_{\beta, q_1,q_2}^2(s) \leq \frac{s^{2q_1}}{T_{q_2}^2((2/\beta) - 1)}
	\leq 4 s^{2q_1} \left( \frac{1 - \sqrt{1-\beta}}{1 + \sqrt{1-\beta}} \right)^{2q_2}
	= 4 s^{2q_1} \delta^{2q_2}
	\quad\text{for $0 \leq s \leq \beta$.}
\end{equation}
The first inequality follows from~\eqref{eqn:Tp-minmax},
and the second follows from~\eqref{eqn:Tp-growth}.
Last, we instate the definition~\eqref{eqn:delta}.

\begin{remark}[The Monomial]
In contrast to~\cite{KW92:Estimating-Largest} and other prior work,
we use products of Chebyshev polynomials with low-degree monomials.
This seemingly minor change leads to results phrased in terms of
the stable rank, rather than the ambient dimension.  This is equivalent
to a bound on $q_1$ iterations of the subspace iteration method,
followed by $q_2$ iterations of the block Krylov method.
\end{remark}

\subsubsection{Chebyshev Polynomials of the Second Kind}

We can define the Chebyshev polynomials of the second kind via the formula
\begin{equation} \label{eqn:Up}
U_p(s) := \frac{1}{2\sqrt{s^2 - 1}} \left[ \left(s + \sqrt{s^2 - 1}\right)^{p+1} - \left(s - \sqrt{s^2 - 1}\right)^{p+1} \right]
	\quad\text{for $s \in \R$ and $p \in \Z_+$.}
\end{equation}
Using the binomial theorem, it is easy to check that this expression coincides
with a polynomial of degree $p$ with real coefficients.  Moreover, when $p$
is an even number, the polynomial $U_p$ is an even function.

We require two properties of the Chebyshev polynomial $U_p$.
First, it satisfies a weighted uniform bound on the unit interval:
\begin{equation} \label{eqn:Up-minmax}
\abs{ \sqrt{1 - s^2} \, U_p(s) } \leq 1
\quad\text{for $\abs{s} \leq 1$.}
\end{equation}
In fact, $2^{-p} U_p$ is the unique monic polynomial that minimizes
the maximum value of the left-hand side of~\eqref{eqn:Up-minmax}
over the interval $[-1,+1]$; see~\cite[Eqn.~(23) et seq.]{KW92:Estimating-Largest}.
The simpler result~\eqref{eqn:Up-minmax} is an immediate consequence of the representation
$$
U_p(s) = \frac{ \sin\big( (p+1) \cos^{-1} (s) \big) }{ \sqrt{1 - s^2} }
\quad\text{for $\abs{s} \leq 1$.}
$$
The latter formula follows from~\eqref{eqn:Up} after we apply de Moivre's theorem
for complex exponentiation.

Second, we can evaluate the polynomial $U_{2p}$ at a specific point:
\begin{equation} \label{eqn:Up-value}
U_{2p}^2\big( \sqrt{1/\smash{\beta}} \big)
	= \frac{\beta \big(1-\delta^{2p+1} \big)^2}{4(1-\beta) \delta^{2p+1}}
	\quad\text{where $0 < \beta \leq 1$.}
\end{equation}
We defined $\delta = \delta(\beta)$ above in~\eqref{eqn:delta}.
The formula~\eqref{eqn:Up-value} is a direct---but unpleasant---consequence of the definition~\eqref{eqn:Up}.

\subsubsection{Second Polynomial Construction}
\label{app:second-poly}

As before, introduce a parameter $\beta \in [0, 1]$.
Choose a nonnegative integer partition $q = q_1 + q_2$.
Consider the polynomial
\begin{equation} \label{eqn:psi-def}
\psi_{\beta,q_1,q_2}(s) := \frac{s^{q_1} U_{2q_2}\big(\sqrt{s/\beta} \big)}{U_{2q_2}\big(\sqrt{1/\beta} \big)}
\quad\text{for $s \in \R$.}
\end{equation}
Since $U_{2q_2}$ is an even polynomial, this expression defines a polynomial $\psi_{\beta,q_1,q_2}$
with degree $q$ and normalized to take the value one at $s = 1$.
We have the bound
\begin{equation} \label{eqn:psi-bound}
(\beta - s) \, \psi_{\beta,q_1,q_2}^2(s)
	\leq \frac{s^{2q_1} \beta}{U_{2q_2}^2(\sqrt{1/\beta})}
	= \frac{4 (1 - \beta) s^{2q_1}  \delta^{2q_2+1}}{\big(1 - \delta^{2q_2+1}\big)^2}
	\quad\text{for $0 \leq s \leq \beta$.}
\end{equation}
The inequality~\eqref{eqn:psi-bound} follows from~\eqref{eqn:Up-minmax},
and the equality follows from~\eqref{eqn:Up-value}.
Last, we instate the definition~\eqref{eqn:delta}.
For $s > \beta$, the polynomial grows very rapidly.

\section{The Error in the Block Krylov Subspace Method}
\label{sec:error-formula}

In this section, we initiate the proof of Theorems~\ref{thm:gapfree} and~\ref{thm:gap}.
Along the way, we establish Proposition~\ref{prop:few-eigs}.
First, we show how to replace the block Krylov subspace by a simple Krylov subspace (with block size one).
Afterward, we develop an explicit representation for the error in the eigenvalue estimate
derived from the simple Krylov subspace.  Finally, we explain how to construct the
simple Krylov subspace so that we preserve the benefits of computing a block Krylov subspace.

The ideas in this section are drawn from several sources.
The strategy of reducing a block Krylov subspace to a simple Krylov subspace
already appears in~\cite{Saa80:Rates-Convergence}, but we use a different
technique that is adapted from~\cite{HMT11:Finding-Structure}.  The kind of
analysis we perform for the simple Krylov method is standard; we have
closely followed the presentation in~\cite{KW92:Estimating-Largest}.

\subsection{Simplifications}

Suppose that the input matrix $\mtx{A}$ is a multiple of the identity matrix.
From the definitions~\eqref{eqn:krylov-matrix},~\eqref{eqn:krylov-subspace}, and \eqref{eqn:eigenvalue-estimate},
it is straightforward to check that the eigenvalue estimate
$\xi_{\max}(\mtx{A}; \mtx{\Omega}; q) = \lambda_{\max}(\mtx{A})$
with probability one for each $q \geq 0$.
Therefore, we may as well assume that $\mtx{A}$ is not a multiple of the identity.

In view of~\eqref{eqn:input-diag}, we may also assume that the input matrix is diagonal
with weakly decreasing entries:
\begin{equation} \label{eqn:A-diag}
\mtx{A} = \diag( a_1, a_2, \dots, a_n )
\quad\text{where}\quad
a_1 \geq a_2 \geq \dots \geq a_n.
\end{equation}
Since $\mtx{A}$ has at least two distinct eigenvalues, we may normalize
the extreme eigenvalues of $\mtx{A}$:
\begin{equation} \label{eqn:normalization}
a_1 = 1
\quad\text{and}\quad
a_n = 0.
\end{equation}
The main results are all stated in terms of affine invariant quantities,
so we have not lost any generality.  These choices help to streamline the proof.

\subsection{Block Krylov Subspaces and Simple Krylov Subspaces}

The first key step in the argument is to reduce the block Krylov subspace to
a Krylov subspace with block size one.  This idea allows us to avoid any
computations involving matrices.
To that end, recall that the block Krylov subspace takes the form
$$
K_q(\mtx{A}; \mtx{\Omega})
	= \range \begin{bmatrix}
	\mtx{\Omega} & \mtx{A\Omega} & \mtx{A}^2 \mtx{\Omega} & \dots & \mtx{A}^q \mtx{\Omega}
	\end{bmatrix}.
$$
In particular, for any vector $\vct{x} \in \range(\mtx{\Omega})$,
$$
K_q(\mtx{A}; \vct{x})
	= \range \begin{bmatrix}
	\vct{x} & \mtx{A}\vct{x} & \mtx{A}^2 \vct{x} & \dots & \mtx{A}^q \vct{x}
	\end{bmatrix}
	\subset K_q(\mtx{A}; \mtx{\Omega}).
$$
Later, we will make a careful choice of the vector $\vct{x}$
so that we do not abandon the benefits of computing the block Krylov subspace.

\subsection{Representation of the Error Using Polynomials}

The next step in the argument is to exploit the close relationship between
Krylov subspaces and polynomial filtering to obtain an explicit representation
of the error in the eigenvalue estimate.
Using~\eqref{eqn:krylov-poly}, we may rewrite the last display in the form
$$
K_q(\mtx{A}; \vct{x})
	= \lspan{} \big\{ \phi(\mtx{A}) \vct{x} : \phi \in \mathcal{P}_q \big\}
	\subset K_q(\mtx{A}; \mtx{\Omega}).
$$
As a consequence of this containment,
$$
\begin{aligned}
\xi_{\max}( \mtx{A}; \mtx{\Omega};q )
	= \max_{\vct{v} \in K_q(\mtx{A}; \mtx{\Omega})} \frac{\vct{v}^* \mtx{A} \vct{v}}{\vct{v}^* \vct{v}}
	\geq \max_{\vct{v} \in K_q(\mtx{A}; \vct{x})} \frac{\vct{v}^* \mtx{A} \vct{v}}{\vct{v}^* \vct{v}}
	= \max_{\phi \in \mathcal{P}_q} \frac{(\phi(\mtx{A}) \vct{x})^* \mtx{A} (\phi(\mtx{A}) \vct{x})}
	{(\phi(\mtx{A}) \vct{x})^*(\phi(\mtx{A}) \vct{x})}.
\end{aligned}
$$
Owing to the normalization~\eqref{eqn:normalization},
the relative error~\eqref{eqn:error-def} in the eigenvalue estimate satisfies
$$
\mathrm{err}(\xi_{\max}(\mtx{A}; \mtx{\Omega}; q))
	= 1 - \xi_{\max}(\mtx{A}; \mtx{\Omega}; q)
	\leq \min_{\phi \in \mathcal{P}_q} \frac{(\phi(\mtx{A}) \vct{x})^* (\Id - \mtx{A}) (\phi(\mtx{A}) \vct{x})}
	{(\phi(\mtx{A}) \vct{x})^*(\phi(\mtx{A}) \vct{x})}.
$$
The fraction is continuous on the set of polynomials $\{ \phi \in \coll{P}_q : \phi(\mtx{A}) \vct{x} \neq \vct{0} \}$,
and it is invariant under scaling of the polynomial $\phi$.
Thus, assuming that the first coordinate of $\vct{x}$ is nonzero,
we may assume that $\phi(1) \neq 0$ and then rescale so that $\phi(1) = 1$.
With the definition $\mathcal{P}_q(1) := \{ \phi \in \mathcal{P}_q : \phi(1) = 1 \}$,
we arrive at
$$
\mathrm{err}(\xi_{\max}(\mtx{A}; \mtx{\Omega}; q))
	\leq \inf_{\phi \in \mathcal{P}_q(1)} \frac{(\phi(\mtx{A}) \vct{x})^* (\Id - \mtx{A}) (\phi(\mtx{A}) \vct{x})}
	{(\phi(\mtx{A}) \vct{x})^*(\phi(\mtx{A}) \vct{x})}.
$$
We remark that this inequality becomes an equality in the case where the block size $\ell = 1$.

Invoke~\eqref{eqn:A-diag} to rewrite this bound in terms of the eigenvalues of $\mtx{A}$:
\begin{equation} \label{eqn:err-poly}
\mathrm{err}(\xi_{\max}(\mtx{A}; \mtx{\Omega}; q))
	\leq \inf_{\phi \in \mathcal{P}_q(1)} \frac{\sum_{i=1}^n X_i^2 \phi^2(a_i)(1 - a_i)}
	{\sum_{i=1}^n X_i^2 \phi^2(a_i)}
	= \inf_{\phi \in \mathcal{P}_q(1)} \frac{\sum_{i>1} X_i^2 \phi^2(a_i)(1 - a_i)}
	{X_1^2 + \sum_{i>1} X_i^2 \phi^2(a_i)}.
\end{equation}
We have introduced the coordinates of the vector $\vct{x} = (X_1, \dots, X_n)$,
and we have used normalization~\eqref{eqn:normalization} to simplify the expression.

\begin{remark}[Multiplicity]
When the maximum eigenvalue has multiplicity $m$ greater than one,
we can group the copies of the eigenvalue together in the denominator
to obtain the larger term $X_1^2 + \dots + X_m^2$.  Pursuing this
argument, we see that the algorithm converges faster when $m > 1$.
\end{remark}

\subsection{When the Matrix has Few Distinct Eigenvalues}

Suppose that the distinct eigenvalues of the input matrix $\mtx{A}$
are $\mu_1 = 1$ and $\mu_2, \dots, \mu_r$ where $r \leq q + 1$.  Consider the polynomial
$$
\phi_0(s) = \prod\limits_{i = 2}^r \frac{s - \mu_i}{1 - \mu_i} \in \mathcal{P}_{r-1}(1) \subset \mathcal{P}_q(1).
$$
This polynomial annihilates each point in the spectrum, except for $\mu_1$.
Setting $\phi = \phi_0$ in~\eqref{eqn:err-poly}, we discover that
$\mathrm{err}(\xi_{\max}(\mtx{A};\mtx{\Omega}; q)) = 0$ with probability one.
This fact appears as Proposition~\ref{prop:few-eigs}.

\subsection{Choosing the Simple Krylov Space}

The next step in the argument is to select a particular vector $\vct{x} \in \range(\mtx{\Omega})$.
Let $\vct{\omega}_1^* \in \R^{\ell}$ denote the first row of the matrix $\mtx{\Omega}$.
Set
$$
\vct{x} = \frac{\mtx{\Omega} \vct{\omega}_1}{\norm{\vct{\omega}_1}} \in \range(\mtx{\Omega}).
$$
The rows of $\mtx{\Omega}$ are statistically independent standard normal vectors,
which are rotationally invariant.  It follows that the entries of
$\vct{x} = (X_1, \dots, X_n)$ are also statistically independent.  Moreover,
\begin{equation} \label{eqn:x-distrib}
X_1 \sim \textsc{chi}(\ell)
\quad\text{and}\quad
X_i \sim \normal(0, 1)
\quad\text{for $i > 1$.}
\end{equation}
We write $\textsc{chi}(\ell)$ for the chi distribution with $\ell$ degrees of freedom.
This choice of $\vct{x}$ ensures that $X_1^2$ is large relative to the other $X_i^2$.
In particular, $X_1 \neq 0$ with probability one.

In a concrete sense, we see that the block Krylov method is an artifice for increasing
the apparent multiplicity of the maximum eigenvalue.
The argument here is inspired by the analysis
in Halko et al.~\cite[Secs.~9, 10]{HMT11:Finding-Structure}.

\begin{remark}[Complex Setting]
When the input matrix is Hermitian and the test matrix follows the complex standard normal
distribution, we obtain the same formulas with $X_1 \sim \textsc{chi}(2\ell)$ and
$X_i \sim \textsc{normal}_{\C}(0, 1)$.  In effect, the multiplicity of
every eigenvalue is doubled.
\end{remark}

\section{Probabilistic Bounds for the Error}
\label{sec:prob-bounds}

This section contains the proof of the probability bounds that appear
in Theorem~\ref{thm:gapfree} and~\ref{thm:gap}.
The argument is based on the bound~\eqref{eqn:err-poly} for the error
and the distributional properties~\eqref{eqn:x-distrib} of the random vector $\vct{x}$.

The proof is inspired by the argument in Kuczy{\'n}ski \& Wo{\'z}niakowski~\cite{KW92:Estimating-Largest},
but our approach is technically easier.  Indeed, they work with a random vector that is uniformly
distributed on the Euclidean unit sphere, which leads to a difficult multivariate integration.
In contrast, our random vector $\vct{x}$ has independent entries, which means
that we only have to compute one-dimensional integrals.

\subsection{A Bound for the Probability}

Let $\eps \in (0, 1)$ be an error tolerance.
Our goal is to control the probability $P_{\eps}$ that the relative error
in the eigenvalue estimate~\eqref{eqn:eigenvalue-estimate} is at least $\eps$.
In other words, we wish to bound
$$
P_{\eps} := \Prob{ \mathrm{err}(\xi_{\max}(\mtx{A}; \mtx{\Omega}; q)) \geq \eps }.
$$
In view of the upper bound~\eqref{eqn:err-poly} for the relative error,
we obtain the estimate
$$
P_{\eps} \leq \Prob{ \inf_{\phi \in \mathcal{P}_q(1)} \frac{\sum_{i > 1} X_i^2 \phi^2(a_i) (1 - a_i)}{X_1^2 + \sum_{i > 1} X_i^2 \phi^2(a_i)} \geq \eps }.
$$
Fix a polynomial $\phi \in \mathcal{P}_q(1)$, to be determined later.  Then rearrange the inequality
in the event:
\begin{equation} \label{eqn:Peps-bd1}
\begin{aligned}
P_{\eps} &\leq \Prob{ \sum_{i > 1} X_i^2 \phi^2(a_i)(1 - a_i) \geq \eps X_1^2 + \eps \sum_{i > 1} X_i^2 \phi^2(a_i) } \\
	&= \Prob{ - \eps X_1^2 + \sum_{i > 1} X_i^2 \phi^2(a_i)(1 - \eps - a_i) \geq 0 } \\
	&\leq \Prob{ - \eps X_1^2 +  \sum_{a_i < 1 - \eps} X_i^2 \phi^2(a_i)(1 - \eps - a_i) \geq 0 } \\
	&=: \Prob{ - \eps X_1^2 + \sum_{i \in I} c_i X_i^2 \geq 0 }.
\end{aligned}
\end{equation}
To reach the third line, we have dropped the nonpositive terms in the sum.  Then we introduced the compact notation
$$
I := \{ i : a_i <  1 - \eps \}
\quad\text{and}\quad
c_i := \phi^2(a_i)(1 - \eps - a_i) > 0.
$$
To continue the argument, we apply some elementary notions
from the theory of concentration of measure.

\subsection{The Laplace Transform Argument}

We invoke the Laplace transform method to convert the probability bound into
an expectation bound.  Introduce a parameter $\theta > 0$, to be chosen later.
Continuing from~\eqref{eqn:Peps-bd1}, we write the probability as the expectation
of a 0--1 indicator function:
$$
\begin{aligned}
P_{\eps} &\leq \Expect \mathbb{1} \left\{ - \eps X_1^2 + \sum_{i \in I} c_i X_i^2 \geq 0 \right\} \\
	&\leq  \Expect \exp\left( - \theta \eps X_1^2 + \sum_{i \in I} \theta c_i X_i^2 \right).
\end{aligned}
$$
To reach the second line, we bound the indicator $\mathbb{1}\{ s \geq 0 \}$ above by the function $s \mapsto \econst^{\theta s}$.
Write the exponential as a product, and invoke the independence of the family $\{ X_i \}$ to obtain
$$
P_{\eps} \leq \Expect \left[ \econst^{- \theta \eps X_1^2} \prod_{i \in I} \econst^{\theta c_i X_i^2}  \right]
	= \left( \Expect \econst^{- \theta \eps X_1^2} \right) \left( \prod_{i \in I} \Expect \econst^{\theta c_i X_i^2} \right) .
$$
The distributional property~\eqref{eqn:x-distrib} implies that $X_1^2$ is a chi-squared variable with $\ell$ degrees of freedom, while $X_i^2$ is a chi-squared variable with one degree of freedom.  Computing the remaining expectations is a standard exercise, which results in the bound
$$
P_{\eps} \leq  \left( 1 + 2 \theta \eps \right)^{-\ell/2}\left( \prod_{i \in I} (1 - 2 \theta c_i) \right)^{-1/2}
\quad\text{when $\theta < (2c_i)^{-1}$ for each $i \in I$.}
$$
We make a coarse estimate to arrive at
$$
P_{\eps} \leq  \left( 1 + 2 \theta \eps \right)^{-\ell/2} \left( 1 - 2 \theta  \sum_{i \in I} c_i \right)^{-1/2}
\quad\text{when $\theta < \big(2 \sum_{i \in I} c_i \big)^{-1}$.}
$$
The last bound follows from repeated application of the numerical inequality
$(1 - s)(1 - t) \geq 1 - (s + t)$, which is valid when $st \geq 0$.

Next, we must identify a suitable value for $\theta$.
It is possible to minimize the probability bound with respect to $\theta$,
but it is more expedient to select
$\theta^{-1} = 4 \sum_{i \in I} c_i$.  This choice yields
\begin{equation} \label{eqn:Peps-penult}
P_{\eps} \leq \sqrt{2} \left( 1 + \frac{\eps/2}{\sum_{i \in I} c_i} \right)^{-\ell / 2}
	= \sqrt{2} \left[ 1 + \frac{\eps/2}{\sum_{a_i < 1 - \eps} \phi^2(a_i) (1 - \eps - a_i)} \right]^{-\ell / 2}.
\end{equation}
It remains to choose a good polynomial $\phi$.
Fortunately, we have already done the work in Section~\ref{app:chebyshev}.

\subsection{Probability Bound without a Spectral Gap}

In this section, we use~\eqref{eqn:Peps-penult} to derive
the probability bound that appears in Theorem~\ref{thm:gapfree}.
Set the parameter $\beta = 1 - \eps$.  Obviously,
$$
\sum_{a_i < 1 - \eps} \phi_1^2(a_i) (1 - \eps - a_i)
	= \sum_{a_i < \beta} \phi_1^2(a_i) (\beta - a_i).
$$
To control the terms in this sum, we need second-kind
Chebyshev polynomials because they satisfy the weighted
uniform bound~\eqref{eqn:Up-minmax}.
For a partition $q =  q_1 + q_2$, consider the polynomial
$$
\phi_1(s) := \frac{s^{q_1} U_{2q_2}\big( \sqrt{s/\beta} \big)}{U_{2q_2}\big(\sqrt{1/\beta}\big)} \in \mathcal{P}_q(1).
$$
According to~\eqref{eqn:psi-bound} and~\eqref{eqn:delta-bd}, this polynomial satisfies
\begin{equation} \label{eqn:my-phi1-bd}
\phi^2_1(s) (\beta - s) \leq \frac{4 \eps s^{2q_1} \delta^{2q_2+1}}{(1 - \delta^{2q_2+1})^2}
\quad\text{for $0 \leq s \leq \beta$}
\quad\text{where}\quad
\delta \leq \econst^{-2\sqrt{\eps}}.
\end{equation}
See Section~\ref{app:second-poly} for further discussion.

Using these facts, we may estimate that
\begin{equation*} %
\begin{aligned}
\sum_{a_i < 1 - \eps} \phi_1^2(a_i) (1 - \eps - a_i)
	&\leq 4 \eps \left( \sum_{a_i < \beta} a_i^{2q_1} \right)
	\delta^{2q_2+1}  \big(1 - \delta^{2q_2+1}\big)^{-2} \\
	&\leq 4 \eps \srank(q_1) \, \delta^{2q_2+1} \big(1 - \delta^{2q_2+1}\big)^{-2}.
\end{aligned}
\end{equation*}
In the last step, we bound the sum in terms of the stable rank~\eqref{eqn:srank}.
We rely on the normalization~\eqref{eqn:normalization} to recognize the stable rank.

Select $\phi = \phi_1$ in our probability bound~\eqref{eqn:Peps-penult}.
Using the last display, we arrive at
$$
P_{\eps} \leq \sqrt{2} \left[ 1 + \frac{ \big( 1 - \delta^{2q_2+1} \big)^2 }{8 \srank(q_1) \delta^{2q_2+1}} \right]^{-\ell/2}.
$$
We can develop a lower bound on the bracket as follows.
$$
1 + \frac{ \big( 1 - \delta^{2q_2+1} \big)^2 }{8 \srank(q_1) \delta^{2q_2+1}}
	\geq 1 + \frac{1 - 2\delta^{2q_2+1}}{8 \srank(q_1)\, \delta^{2q_2+1}}
	\geq \frac{1}{8 \srank(q_1)\,\delta^{2q_2+1}}.
$$
The last inequality follows from the fact that $\srank(q_1) \geq 1$
because $\mtx{A}$ is not a multiple of the identity.
Combining the last two displays, we obtain
\begin{equation} \label{eqn:prob-gapfree}
P_{\eps} \leq
	\sqrt{2} \left[ 8 \srank(q_1) \, \delta^{2q_2+1} \right]^{\ell / 2}
	\leq \sqrt{2} \left[ 8 \srank(q_1) \cdot \econst^{-2(2q_2+1) \sqrt{\eps}} \right]^{\ell/2}.
\end{equation}
The final relation is a consequence of the bound for $\delta$ in~\eqref{eqn:my-phi1-bd}.
This is the required statement.

\subsection{Probability Bound with a Spectral Gap}

Now, we use~\eqref{eqn:Peps-penult} to establish the probability bound that appears in Theorem~\ref{thm:gap}.
Recall that the spectral gap $\gamma$ is defined in~\eqref{eqn:spectral-gap}.
This time, set $\beta = 1 - \gamma$.  Since $1 - \gamma$ is the first eigenvalue smaller than one
and $\eps > 0$,
$$
\sum_{a_i \leq 1 - \eps} \phi_{2}^2(a_i) (1- \eps - a_i)
	\leq \sum_{a_i \leq 1 - \gamma} \phi_{2}^2(a_i).
$$
To control the sum, we need first-kind
Chebyshev polynomials because they satisfy the
uniform bound~\eqref{eqn:Tp-minmax}.
For a partition $q = q_1 + q_2$, construct the polynomial
\begin{equation} \label{eqn:my-phi2}
\phi_2(s) := \frac{s^{q_1} T_{q_2}((2/\beta)s - 1)}{T_{q_2}((2/\beta) - 1)} \in \mathcal{P}_q(1).
\end{equation}
According to~\eqref{eqn:delta-bd} and~\eqref{eqn:phi-bound}, this polynomial satisfies
\begin{equation} \label{eqn:my-phi2-bd}
\phi_{2}^2(s) \leq 4 s^{2q_1} \econst^{-4q_2\sqrt{\gamma}}
	\quad\text{for $0 \leq s \leq \beta$.}
\end{equation}
See Section~\ref{app:first-poly} for more details.

Using these facts, we calculate that
$$
\begin{aligned}
\sum_{a_i \leq 1 - \eps} \phi_{2}^2(a_i) (1- \eps - a_i)
	&\leq 4 \left( \sum_{a_i \leq 1 - \gamma} a_i^{2q_1} \right) \econst^{-4q_2\sqrt{\gamma}}
	\leq 4 \srank(q_1) \cdot \econst^{-4q_2\sqrt{\gamma}}.
\end{aligned}
$$
We have invoked~\eqref{eqn:normalization} to identify the stable rank~\eqref{eqn:srank}.

Instantiate the probability estimate~\eqref{eqn:Peps-penult} with $\phi = \phi_{2}$,
and substitute in the last display to obtain
\begin{equation} \label{eqn:prob-gap}
P_{\eps} \leq \sqrt{2} \left[ 1 + \frac{\eps}{8 \srank(q_1) \cdot \econst^{-4q_2 \sqrt{\gamma}}} \right]^{-\ell/2}
	\leq \sqrt{2} \left[ \frac{8 \srank(q_1)}{\eps} \cdot \econst^{-4q_2 \sqrt{\gamma}} \right]^{\ell/2}.
\end{equation}
We have used the numerical inequality $(1 + 1/s)^{-1} \leq s$, valid for $s > 0$.
This is the advertised result.

\section{A Bound for the Expected Error without a Spectral Gap}
\label{sec:expect-bound-gapfree}

In this section, we establish the expectation bound that appears in Theorem~\ref{thm:gapfree}.
To obtain this result, we simply
integrate the probability bound~\eqref{eqn:prob-gapfree}.  Surprisingly,
this approach appears to be more effective than a direct computation of
the expected error.  This insight yields a better expected
error bound than the one obtained in~\cite{KW92:Estimating-Largest}.

\subsection{Computing the Expectation}

We may express the expectation of the relative error as an integral:
$$
E := \Expect \mathrm{err}(\xi_{\max}(\mtx{A}; \mtx{\Omega};q))
	= \int_0^1 \Prob{ \mathrm{err}(\xi_{\max}(\mtx{A}; \mtx{\Omega};q)) \geq \eps} \idiff{\eps}
	= \int_0^1 P_{\eps} \idiff{\eps}.
$$
The limits of the integral follow from the fact that the relative error
falls in the interval $[0, 1]$.
We split the integral at a value $c > 0$, to be determined later.  Then make the estimates
$$
E \leq c + \int_c^1 P_{\eps} \idiff{\eps}
	\leq c + \sqrt{2}\, (8 \srank(q_1))^{\ell/2} \int_c^\infty \econst^{-(2q_2+1) \ell \sqrt{\eps}} \idiff{\eps}.
$$
To obtain the first inequality, we use the trivial bound $P_{\eps} \leq 1$.
The second inequality is a consequence of~\eqref{eqn:prob-gapfree}.
The remaining integral can be calculated by changing the variable and integrating by parts.
Indeed,
$$
\int_c^\infty \econst^{-p \sqrt{\eps}} \idiff{\eps}
	= 2 \left( \frac{\sqrt{c}}{p} + \frac{1}{p^2} \right) \econst^{-p \sqrt{c}}
	\quad\text{for $p > 0$.}
$$
Together, the last two displays yield
$$
E \leq c + 2 \sqrt{2} (8 \srank(q_1))^{\ell/2} \left( \frac{\sqrt{c}}{(2q_2+1)\ell} + \frac{1}{(2q_2+1)^2 \ell^2} \right)
	\econst^{- (2q_2+1)\ell \sqrt{c}}.
$$
Now, select the (optimal) value
$$
c = \left( \frac{\ell^{-1} \log 2 + \log(8 \srank(q_1))}{2(2q_2+1)} \right)^2.
$$
Combine the last two displays to reach
$$
E \leq \left( \frac{\ell^{-1}(2 + \log 2) + \log(8 \srank(q_1))}{2(2q_2+1)} \right)^{2}
$$
Bound the numerical constant by 2.70 to complete the proof.

\section{A Bound for the Expected Error with a Spectral Gap}
\label{sec:expect-bound-gap}

Last, we establish the expectation bounds for the relative
error that appear in Theorem~\ref{thm:gap}.
In this case, we achieve better results by a direct computation,
rather than by integrating the probability bound~\eqref{eqn:prob-gap}.
These arguments are inspired by the approach
in~\cite{KW92:Estimating-Largest}, but our
computations are technically easier because 
the random vector $\vct{x}$ has independent entries.

\subsection{Form of the Expected Error}

Fix a polynomial $\phi \in \mathcal{P}_q(1)$.  Take the expectation of
the error bound~\eqref{eqn:err-poly}:
\begin{equation*} %
\Expect \mathrm{err}(\xi_{\max}(\mtx{A}; \mtx{\Omega}; q))
	\leq \Expect\left[ \frac{\sum_{i > 1} X_i^2 \phi^2(a_i)(1-a_i)}{X_1^2 + \sum_{i > 1} X_i^2 \phi^2(a_i)} \right].
\end{equation*}
Set the parameter $\beta = 1 - \gamma$.
By the definition~\eqref{eqn:spectral-gap} of the spectral gap, each eigenvalue
of $\mtx{A}$ that exceeds $\beta$ equals the maximum eigenvalue $a_1 = 1$.  Therefore,
\begin{equation} \label{eqn:unexpected}
\Expect \mathrm{err}(\xi_{\max}(\mtx{A}; \mtx{\Omega}; q))
	\leq \Expect\left[ \frac{\sum_{a_i \leq \beta} X_i^2 \phi^2(a_i)}{X_1^2 + \sum_{a_i \leq \beta} X_i^2 \phi^2(a_i)} \right].
\end{equation}
By independence, we may compute the expectation with respect to $X_1$,
holding $X_i$ fixed for each $i$ where $a_i \leq \beta$.
The computation of this integral depends on the block size $\ell$.

\subsection{Error Bound for Block Size $\ell \geq 3$}

We begin with the case $\ell \geq 3$. %
Use the fact that $X_1 \sim \textsc{chi}(\ell)$ to compute that
$$
\Expect \left[ \frac{1}{X_1^2 + c} \right]
	= \frac{1}{2} \econst^{c/2} \int_1^\infty s^{-\ell/2} \econst^{-cs/2} \idiff{s}
	\leq \frac{1}{(\ell - 2) + c}
	\quad\text{for $\ell \geq 3$ and $c \geq 0$.}
$$
The first relation depends on a standard identity for the partial gamma function~\cite[Sec.~8.6.4]{NIST10:Handbook-Mathematical}.
The second relation is a classic bound for the exponential integral due to
Hopf~\cite[p.~26]{Hop34:Mathematical-Problems}; see~\cite{Gau59:Exponential-Integral} or~\cite[Sec.~5.1.19]{AS64:Handbook-Mathematical}.

With $c = \sum_{a_i \leq \beta} X_i^2 \phi^2(a_i)$, the last two displays imply that
\begin{equation*} \label{eqn:unsplit-l3}
\Expect \mathrm{err}(\xi_{\max}(\mtx{A}; \mtx{\Omega};q))
	\leq \Expect \left[ \frac{\sum_{a_i\leq\beta} X_i^2 \phi^2(a_i)}{(\ell - 2) + \sum_{a_i\leq \beta} X_i^2 \phi^2(a_i)} \right].
\end{equation*}
The function $t \mapsto t / (s + t)$ is concave, so Jensen's inequality
allows us to draw the expectation inside the function to reach
\begin{equation*} \label{eqn:unsplit-l4}
\Expect \mathrm{err}(\xi_{\max}(\mtx{A}; \mtx{\Omega};q))
	\leq \frac{\sum_{a_i\leq\beta} \phi^2(a_i)}{(\ell - 2) + \sum_{a_i\leq \beta} \phi^2(a_i)}.
\end{equation*}
Indeed, $\Expect X_i^2 = 1$ because $X_i$ is standard normal for each $i > 1$.

Introduce the polynomial $\phi = \phi_2$ from~\eqref{eqn:my-phi2} into the last display.
Using the upper bound~\eqref{eqn:my-phi2-bd} and the fact that $t \mapsto t/(s+t)$ is increasing,
we arrive at
$$
\Expect \mathrm{err}(\xi_{\max}(\mtx{A}; \mtx{\Omega}))
	\leq \frac{4 \srank(q_1) \cdot \econst^{-4q_2 \sqrt{\gamma}}}{(\ell - 2) + 4 \srank(q_1) \cdot \econst^{-4q_2 \sqrt{\gamma}}}.
$$
This is the advertised result for $\ell = 3$ in Theorem~\ref{thm:gap}.

\subsection{Error Bound for Block Size $\ell = 2$}

Now, assume that the block size $\ell = 2$.
In this case, $X_1 \sim \textsc{chi}(2)$.  Therefore, it holds that
$$
\Expect \left[ \frac{1}{X_1^2 + c} \right]
	= \frac{1}{2} \econst^{c/2} \int_1^\infty s^{-1} \econst^{-cs/2} \idiff{s}
	\leq \frac{1}{2} \log\left(1 + \frac{2}{c}\right)
	\quad\text{for $c > 0$.}
$$
The first relation is an immediate consequence of the definition of
the chi-square density and a change of variables.
The second relation is a classic bound for the exponential integral~\cite[Sec.~5.1.20]{AS64:Handbook-Mathematical}.

The rest of the argument follows the same path as in the case $\ell \geq 3$.
With $c = \sum_{a_i \leq \beta} X_i^2 \phi^2(a_i)$, we combine~\eqref{eqn:unexpected}
with the last display to obtain
\begin{equation*} \label{eqn:unsplit-l2}
\Expect \mathrm{err}(\xi_{\max}(\mtx{A}; \mtx{\Omega}))
	\leq \Expect \left[ \frac{ \sum_{a_i \leq \beta} X_i^2 \phi^2(a_i) }{2}
	\cdot \log\left(1 + \frac{2}{\sum_{a_i \leq \beta} X_i^2 \phi^2(a_i)} \right) \right].
\end{equation*}
The function $t \mapsto t \log(1 + 1/t)$ is concave and increasing, so Jensen's inequality
yields
$$
\Expect \mathrm{err}(\xi_{\max}(\mtx{A}; \mtx{\Omega}))
	\leq \frac{\sum_{a_i \leq \beta} \phi^2(a_i)}{2}
	\cdot \log\left(1 + \frac{2}{\sum_{a_i \leq \beta} \phi^2(a_i)} \right). 
$$
Select the polynomial $\phi = \phi_2$ from~\eqref{eqn:my-phi2},
and invoke the bound~\eqref{eqn:my-phi2-bd} to arrive at
$$
\begin{aligned}
\Expect \mathrm{err}(\xi_{\max}(\mtx{A}; \mtx{\Omega}))
	&\leq \frac{1}{2} \cdot 4 \srank(q_1) \econst^{-4 q_2 \sqrt{\gamma}}
	\cdot \log\left(1 + \frac{2}{4\srank(q_1) \econst^{-4 q_2 \sqrt{\gamma}}} \right). %
\end{aligned}
$$
This is the desired outcome for block size $\ell = 2$  in Theorem~\ref{thm:gap}.

\subsection{Error Bound for Block Size $\ell = 1$}

Finally, we complete the analysis for the case where the block size $\ell = 1$.
Use the fact that $X_1 \sim \textsc{chi}(1)$ to see that
$$
\Expect \left[ \frac{1}{X_1^2 + c} \right]
	= \frac{1}{\sqrt{c}} \econst^{c/2} \int_{\sqrt{c}}^\infty \econst^{-s^2/2} \idiff{s}
	\leq \sqrt{\frac{2\pi}{c}}. %
$$
The first relation follows from~\cite[Sec.~8.6.4]{NIST10:Handbook-Mathematical} after a change of variable.
The second relation is a well-known tail bound for a standard normal random variable.
Continuing as we have done before,
$$
\Expect \mathrm{err}(\xi_{\max}(\mtx{A}; \mtx{\Omega}))
	\leq \sqrt{2\pi} \Expect \left[ \sum_{a_i \leq \beta} X_i^2 \phi^2(a_i) \right]^{1/2}
	\leq \sqrt{2\pi} \left[ \sum_{a_i \leq \beta} \phi^2(a_i) \right]^{1/2}.
$$
Insert the bound~\eqref{eqn:my-phi2-bd} for the polynomial $\phi = \phi_2$ to arrive at
$$
\Expect \mathrm{err}(\xi_{\max}(\mtx{A}; \mtx{\Omega}))
	\leq \sqrt{8 \pi \srank(q_1) } \cdot \econst^{-2 q_2 \sqrt{\gamma}}.
$$
Of course, the relative error is always bounded above by one.
This is the result stated in Theorem~\ref{thm:gap}.

\section*{Acknowledgments}

This paper was produced under a subcontract to ONR Award 00014-16-C-2009.
The author thanks Shaunak Bopardikar, Petros Drineas, Ilse Ipsen, Youssef Saad,
Fadil Santosa, and Rob Webber for helpful discussions and feedback.
Mark Embree provided very detailed, thoughtful comments
that greatly improved the quality of the paper.

\bibliographystyle{myalpha}
\newcommand{\etalchar}[1]{$^{#1}$}

\end{document}